\documentclass[11pt]{amsart}
\usepackage{amsmath,amssymb,amsthm,tikz, mathtools}
\usepackage{comment}

\RequirePackage{color}
\RequirePackage[colorlinks, urlcolor=my-blue,linkcolor=my-red,citecolor=my-green]{hyperref}
\definecolor{my-blue}{rgb}{0.0,0.0,0.6}
\definecolor{my-red}{rgb}{0.5,0.0,0.0}
\definecolor{my-green}{rgb}{0.0,0.5,0.0}
\definecolor{nicos-red}{rgb}{0.75,0.0,0.0}
\definecolor{light-gray}{gray}{0.6}
\definecolor{really-light-gray}{gray}{0.8}
\definecolor{sussexg}{rgb}{0.0,0.5,0.5}
\definecolor{sussexp}{rgb}{0.5,0.0,0.5}

\usepackage[numbers,sort]{natbib}

\usepackage{nicefrac}

\newtheorem{theorem}{\sc \color{sussexg}Theorem}[section]

\newtheorem{corollary}[theorem]{\sc \color{nicos-red} Corollary}

\numberwithin{equation}{section}
\theoremstyle{remark}
\newtheorem{remark}[theorem]{\color{sussexp} Remark}

\newcommand{\be}{\begin{equation}}
\newcommand{\ee}{\end{equation}}

\def\bE{\mathbb{E}}
\def\bN{\mathbb{N}}
\def\bP{\mathbb{P}}

\def\bR{\mathbb{R}}

\def\R{\bR}
\def\N{\bN}
\def\W{{\bf W}}

\def\fS{{\bf S}}

\def\E{\bE}
\def\P{\bP} 

\definecolor{darkgreen}{rgb}{0.0,0.5,0.0}
\definecolor{darkblue}{rgb}{0.0,0.0,0.3}
\definecolor{nicosred}{rgb}{0.65,0.1,0.1}
\definecolor{light-gray}{gray}{0.7}
\allowdisplaybreaks[1]

\title{A restless Time-Fractional Multiclass Queue}
\author{Nicos Georgiou}
\address{Nicos Georgiou\\ University of Sussex\\ 
Mathematics Department\\ Pevensey II\\ Falmer campus, BN19QH} 
\email{n.georgiou@sussex.ac.uk}

\author{Enrico Scalas}
\address{Enrico Scalas\\ University of Rome 1-La Sapienza\\ 
Statistics Department} 
\email{enrico.scalas@uniroma1.it}

\author{Vladislav Vysotsky}
\address{Vladislav Vysotsky \\ University of Sussex\\ 
Mathematics Department\\ Pevensey II\\ Falmer campus, BN19QH} 
\email{v.vysotskiy@sussex.ac.uk}

\keywords{Mittag-Leffler queues, multiclass queue, fractional queue, fractional Poisson process}
\subjclass[2020]{Primary 60K25, 60F17. Secondary 60G55, 60G51}

\date{\today}

\allowdisplaybreaks

\begin{document}

\title[Fractional Multiclass queues]{A Restless Time-Fractional Multiclass Queue}

\begin{abstract}
We study a single-server priority queue with a finite number of classes, in which the arrivals follow a fractional Poisson process of index $\alpha \in (0,1]$ and the service completions are triggered by an independent fractional Poisson process of index $\beta \in (0,1]$. Each of the customers arriving is assigned at random to one of the priority classes. This assignment is independent of the rest of the system and follows a fixed probability distribution. 

Using a time-change representation of a fractional Poisson process, we first give a multinomial thinning decomposition: the total number of arrivals in each class are independent standard Poisson processes of appropriate intensities, time-changed by a common independent random clock that is the inverse of an $\alpha$-stable subordinator. This yields a process-level law of large numbers and a functional central limit theorem for the process of arrivals.

For the queueing system itself, we identify process-level scaling limits for the cumulative and individual queue lengths of the classes. We also prove that the queue gets empty infinitely often when $\alpha \le \beta$, which does include the critical case $\alpha = \beta$. A final example shows how the model can be extended to a continuum of classes.
\end{abstract}

\keywords{Mittag-Leffler queues, multiclass queue, fractional queue, fractional Poisson process}

\maketitle

\tableofcontents

\section{Introduction and preliminaries}  

The theory of queues gives foundational mathematical tools for rigorous modelling of congestion across different disciplines including computing, communications, logistics, and healthcare. Classical models such as M/M/1 and their GI/GI/1 generalisations are analytically tractable and offer insight, as well as genuine modelling methodologies for practitioners, through a vast body of literature \cite{ChenYao2001,Asmussen2003,Robert2003}.
In the studies of such models, the common assumptions are the memoryless Markovian properties or the light-tailness of distributions of the service times or arrival times. The performance and stability of such queues are essentially governed by the moments of these times. Of  particular importance is the notion of `traffic intensity',  through which the stability properties of the queue such as  recurrence, transience, laws of large numbers or fluctuations of the queue length are completely characterised.  \

Empirical evidence, however, shows that the real arrivals and services are often bursty and long-range dependent (for example the autocorrelation of the queue length could decay as a power law),
as opposed to the classical GI/GI/1 queues.

This creates a need to develop a theory for queues in which the distributions of services or arrivals (or both) create long-range effects. One way to construct such queues is to enforce the inter-arrival or service times to be heavy-tailed, in the sense of having infinite mean. This behaviour has already been observed in examples \cite{RabertoScalasMainardi2002, 
SabatelliKeatingDudleyRichmond2002}. 
Such changes 
alter the behaviour of all quantities of interest in the queue, such as system delays, tail asymptotics for waiting times, order of queue lengths, and busy periods as the process evolves in time
\cite{LelandTaqquWillingerWilson1994,PaxsonFloyd1995,FossKorshunovZachary2011}. 
 In such queueing examples the well-trodden traffic intensity approach loses its meaning and one must characterise regimes via scaling limits and reflected-input structures rather than expectations of arrival and service times \cite{AscioneCaputo, ButtGeorgiouScalas2022}.

The easiest way to obtain non-Markovian queues (i.e.\ non-M/M/1) is by changing the distribution of the arrival or the service times to something non-exponential. A standard choice here are random times $X$ with the tail behaviour $\P\{ X >x \} \sim \ell(x) x^{-\theta}$ for some constant $\theta>0$ and a function $\ell$ that is slowly varying at infinity. Among many types of such  distributions is the Mittag-Leffler distribution of index $\theta \in (0,1)$, or ML($\theta$) in short, for which $\ell(x) \equiv 1/\Gamma(1-\theta)$ is constant and the mean is infinite \cite{GorenfloKilbasMainardiRogosin2014}. This heavy-tailed distribution has a broad mathematical interest due to its tractable analytical properties \cite{Pillai1990}.  

There are two basic properties that make it particularly useful in applications to probability. The first one is that the counting process for a renewal process with the ML($\theta$) inter-event times 
has the same distribution as a classical Poisson process time-changed by an independent stochastic process called the inverse $\theta$-stable subordinator. This property lends the name (time-) \emph{fractional Poisson Process} (FPP) of index $\theta \in (0,1]$ to this counting process. This time-change representation was especially useful for establishing limit theorems for FPPs and showcasing properties that arise from the non-Markovianity of the FPPs, including tails, moments and the temporal evolution \cite{Laskin2003,MeerschaertNaneVellaisamy2011,MeerschaertSikorskii2012}. 

The second important property is that the forward Kolmogorov equations for the classical Poisson process remain valid for the FPP of index $\theta$ if the usual derivative of the transition probabilities is replaced by the Caputo derivative of order $\theta$. 

The family of Mittag-Leffler distributions extends to the case $\theta = 1$. For this particular case the distribution does not exhibit power-law tails but
it coincides with the exponential distribution, so the FPP becomes a usual Poisson process, and the fractional Kolmogorov equations turn into the classical ones since the Caputo derivative of order $\theta=1$ is the usual derivative. For this reason, the ML distribution is considered a natural heavy-tailed generalisation of the exponential distribution. This observation is useful when one wants to change a model that involves exponential distribution or a Poisson process to make this model heavy-tailed.

Let us discuss what this change might give in the context of modifying the classical M/M/1 queue with, say, equal arrival and service times rates and FIFO discipline.  The M/M/1 queue can be represented in the following four equivalent ways using the same underlying discrete time Markov chain on $\{0,1, \ldots\}$ that chooses between arrivals and departures on each step.

{\bf Model 1:} At the event times of a standard Poisson process, the queue length (i) increases by $1$ with probability $1/2$, or (ii) with probability $1/2$, either decreases by $1$ if at least one customer is present, or stays~$0$ if there are no customers.

\medskip

{\bf Model 2:} Consider two independent sequences of independent Exp(1/2) times $\{A_i\}_{i \ge 1}$ and $\{B_i\}_{i \ge 1}$. Consider the sequence of independent times $T_i = A_i \wedge B_i \sim {\rm Exp}(1)$ to form a standard Poisson process with the inter-event times $T_i$. The queue length changes at these times so that a customer arrives (+1) at the $i$-th event if $A_i < B_i$, otherwise a customer leaves ($-1$) or the queue stays empty, as above.
\medskip

{\bf Model 3:} Use two independent Poisson(1/2) processes $N_a$ and $N_d$. Whenever $N_a$ has an event, then a customer arrives. Whenever $N_d$ has an event, a customer leaves ($-1$) or the queue stays empty.
\medskip

{\bf Model 4:} Use one Poisson(1/2) process $N_a$ for arrivals. The first customer in the queue exits after an independent $\rm Exp(1)$ service time passes. If the queue is empty, no service times run. 

One of these models shall be chosen if one wants to time-change the Poisson processes above to FPPs or the exponential distributions above to ML ones. Each choice leads to a different fractional queueing model, which  has its own type of behaviour and requires a different mathematical toolbox to tackle it. 

The fractional generalisation of Model 1 was studied most extensively because it provides an example of a time-fractional birth and death chain \cite{CahoyPolitoPhoha2015}. When a single FPP drives the dynamics, the fractional Kolmogorov equations for the queue length are readily available and an array of analytical techniques is available to study the system rigorously.
Similarly, more general queueing models with a single time change can be studied,
 such as fractional Erlang queues \cite{OrsingherPolito2011, 
AscioneLeonenkoPirozzi2020SPA,
AscioneLeonenkoPirozzi2021Chapter}. 

The fractional versions of Models 2 and 3 appeared in \cite{ButtGeorgiouScalas2022}. The former one is called the {\it fractional renewal queue} and is obtained by having independent ML times competing for the jump events. The latter one is  called the {\it restless Mittag-Leffler queue} \cite{AscioneCaputo}, where the two Poisson processes above are substituted by two independent FPPs. As discussed in the following, this model can have wasted service completions since the departure FPP runs even when no customers are present in the system. 
The fractional version of Model~4, developed in \cite{AscioneCaputo}, can be interpreted as a GI/GI/1 queue since its service times 
are now ML-distributed. 

There are several differences between these queueing systems that have been highlighted in both \cite{ButtGeorgiouScalas2022} and \cite{AscioneCaputo}. For example in fractional Model 2 there are indices for the competing ML distributions that make the inter-event times have finite mean. In \cite{AscioneCaputo} the scaling limits for the queue length in all these models were obtained, and then used to characterise their heavy-traffic behaviour.    

\subsection{Contributions and quick summary of results.}
In the current article we 
focus on a multi-class queueing model where both arrivals and service completions are event times of two independent fractional Poisson processes, which makes this the multiclass generalisation of fractional Model 3 above. 

In our multi-class single-server priority queue, the arrivals are generated by {\it thinning (or marking)} the jump times of an FPP of index $\alpha \in (0,1]$. Namely, each arrival is assigned at random to one of $K$ priority classes, according to some fixed probabilities $p_1,\dots,p_K$ and independently of the rest of the system. The service completions are generated by an independent FPP of index $\beta \in (0,1]$. Each jump event of this FPP triggers the immediate departure from the head of the highest-priority nonempty class, if such a non-empty class exists. When all classes are empty, the jump event has no effect. 
Therefore, a jump time
of the departure FPP is not necessarily a service time. From a modelling perspective, this means that the server is actually working even if there are no customers in the queue, in which case the service during the whole empty period could potentially be wasted. Even if the first customer comes in just before a service completion, they will still exit when the departure process ticks, as part of the service was completed before their arrival. Since the fractional version of Model 3 above is called the restless Mittag-Leffler queue (see \cite{AscioneCaputo}), in this article we study the \emph{restless multiclass Mittag-Leffler queue.}

Our first contribution is an explicit structural representation of the class arrivals via a time-changed multinomial thinning of a standard Poisson process. This results in a multivariate process whose marginals, which count arrivals into the classes, are {\it dependent} FPPs represented by {\it independent} Poisson processes time-changed by a \emph{common} fractional clock. For example this immediately allows for calculations for the cross-covariances between the class arrivals and shows that they are positive. This is interesting since normally one would expect that if one class has a high count of marks, then the rest should have less marks, which tacitly suggests that covariances between classes are negative. And indeed, conditional on the total count of marks up to a given time, they are. However, unconditionally, the high variance of the original, unmarked FPP switches this intuition around. In the case where the FPP index is 1, the covariances become 0, and this result reduces to the classical property of Poisson processes that the marked processes are independent Poisson processes.  

Our thinned process is an instance of an existing multivariate fractional counting model \cite{BeghinMacci2016,BeghinMacci2017}. However, in our construction we \emph{derive} the dependence from marking the counting process rather than start with it from the definition. Moreover, the method of proof does not depend on the form of the time change; same methodology will work under very minor regularity conditions. 
A general consideration of thinning of FPPs and of the other counting processes can be found in \cite{mainardi2004}. A particular observation there uses a result from \cite{gnedenko1968}, where the authors construct a sequence of thinnings of heavy-tailed renewal processes, and prove that a limiting distribution of the inter-arrival times exists. Therefore the weak limit of their thinning sequences was also a renewal process. The distribution of the limiting inter-event times was recognised as the Mittag-Leffler distribution in \cite{mainardi2004} and showed that, up to a scaling, the FPP is a `stable point' of the thinning procedure of \cite{gnedenko1968}. This behaviour is further explored in our article, as the arrival process in each of the priority classes is a thinned FPP. 

Second, we prove a functional law of large numbers and a functional central limit theorem for the class arrivals using the multinomially thinned representation. These results provide process-level approximations that appear to be new for thinned fractional models and are directly usable in queueing limits.

Third, we analyze a \emph{multiclass} single-server queue driven by the above arrivals. We obtain scaling limits for aggregated and per-class queue lengths, and classify recurrence and transience of the queue. The effect of heavy tails in both services and arrivals (i.e.\ $\alpha<1$, $\beta<1$) is manifested in the scaling limits by a time-change of the classical results for the M/M/1 queue (with $\alpha=\beta=1$), and also in non-typical recurrence- and transience-type behaviours. Since the process is not Markovian, the notion of ``recurrence'' investigated in this article refers to whether the queue empties infinitely often with probability $1$ and ``transience'' refers to whether the limsup of the queue length is infinity with probability $1$.     

{\bf Structure of the article:}
In Sections \ref{subsec:bg} and \ref{subsec:tc} below 
one can find the necessary background on Mittag-Leffler random variables, FPPs and their time-change representations with inverse stable subordinators. For a more detailed literature see \cite{GorenfloKilbasMainardiRogosin2014}, and for a recent survey see \cite{VanMieghem2020ML}. 
In Section \ref{subsec:model} the restless fractional multiclass queue is introduced. 
Section \ref{sec:results} presents our results for the model, separated into the results for the multivariate arrivals process and the ones for the queueing system itself. 
The proofs for the thinned FPP are in Section \ref{sec:thinproof} while the proofs for the queue itself can be found in Section \ref{sec:qproofs}. Finally, in Section \ref{sec:doubleauction} there is an example that applies the results of the article in a queue with 
a continuum of classes.

{\bf Acknowledgments:} NG and VV acknowledge the support of the Dr Perry James (Jim) Browne Research Center at the Department of Mathematics, University of Sussex. NG acknowledges financial support provided by Sapienza University of Rome through the program Professori Visitatori 2025.

ES was partially supported by the National Recovery and Resilience Plan (NRRP), Mission 4, Component 2, Investment 1.1, Call for tender No. 104 published on 2.2.2022 by the Italian Ministry of University and Research (MUR), funded by the European Union – NextGenerationEU– Project Title “Non–Markovian Dynamics and Non-local Equations” – 202277N5H9 - CUP: D53D23005670006 - Grant Assignment Decree No. 973 adopted on June 30, 2023, by the Italian Ministry of University and Research (MUR). ES also acknowledges financial support provided by Sapienza University of Rome
000317 24 RICERCA UNIV 2023 PROG MEDI SCALAS - RICERCA ATENEO 2023 - SCALAS PROGETTI MEDI. The title of the project is ``Approximation of stochastic processes by means of sums of random telegraph processes”.

The authors would like to thank the anonymous referees for their comments and suggestions. NG would also like to thank Giacomo Ascione for discussions on fractional queueing systems and for early access to the preprint \cite{AscioneCaputo}.

\subsection{Fractional Poisson point process} 
\label{subsec:bg}

The two-parameter Mittag-Leffler function $E_{\beta, \zeta}(z)$ is defined by the power series 
\[
E_{\theta, \zeta}(z) = \sum_{\ell = 0}^\infty \frac{z^{\ell}}{\Gamma(\zeta + \theta \ell)},
\]
where $\theta>0$, $\zeta>0$, $z\in \mathbb{C}$. 
We will mostly require $\zeta = 1$, in which case we just write 
\[
E_{\theta, 1}(z) : = E_{\theta}(z).
\]
Note that $E_{1}(z)=e^z$, as seen from the Maclaurin series for the exponential function. 

The function 
\[
F_{\theta}(x) = 1 - E_{\theta}(-x^{\theta}), \qquad  x\ge 0,
\]
turns out to be the cumulative distribution function (CDF) of a non-negative random variable $X^\theta$ when $\theta \in (0,1]$ (see~\cite{Pillai1990}). For such $\theta$ we call $F_{\theta}$ the Mittag-Leffler cumulative distribution function. It is easy to check that it has density given by 
\[
f_{\theta}(x) = x^{\theta-1}E_{\theta, \theta}(-x^\theta), \qquad x >0.
\] 
In the particular case $\theta =1$ this is the standard exponential density. The reader interested in the derivatives of the Mittag-Leffler function and their numerical approximations can consult \cite{biolek2025}.

We now consider the Mittag-Leffler renewal process, that is a sequence of random variables of the form $\sum_{k=1}^n X^\theta_k$, $n \in \N$, where $X_k^\theta$ are i.i.d.\ inter-event times with the Mittag-Leffler distribution function $F_\theta$ for $\theta \in (0,1]$. The fractional Poisson process with index $\theta$, denoted by $\{N^\theta(t)\}_{t \ge 0}$,  counts the number of events of the Mittag-Leffler renewal process on the time interval $[0, t]$. That is,
\be \label{eq:renewaldef}
N^\theta(t)= n \qquad \Longleftrightarrow \qquad \sum_{k=1}^n X^{\theta}_k \le t < \sum_{k=1}^{n+1} X^{\theta}_k, 
\ee
where $n =0,1, \ldots$ This also implies that $N^\theta(0)=0$ and that 0 is a renewal point. We write 
\[
N^\theta(\cdot) \sim \text{FPP}(\theta).
\]
The single-time marginals of $N^\theta$ are  given by 
(see e.g. \cite{scalas2004})
\begin{equation}
\label{eq:fpd}
 \mathbb{P}(N^\theta(t)=n)= \frac{t^{\theta n}}{n!} E_\theta^{(n)}(-t^\theta),
\end{equation} 
where $E_\theta^{(n)}(-t^\theta)$ stands for  the $n$-th derivative of the function $E_\theta (z)$ evaluated at $z = -t^\theta$. 

For $\theta=1$, $X^1_k$ are standard exponential random variables, and therefore $N^1(t)$ is a standard Poisson process of rate $1$.
It is a Markov process and a Lévy process. We stress that $N^\theta(t)$ is neither a Markov process nor a Lévy process when $\theta \in (0,1)$.

We can add time-dilation by scaling the Mittag-Leffler inter-arrival random variables: for any $\lambda>0$, put
\be\label{eq:scaledML}
X^{\theta, \lambda}_k = \frac{1}{\lambda}X^\theta_k. 
\ee
These random variables have scaled density and CDF 
\[
f_{\theta,\lambda}(t) = \lambda^{\theta}t^{\theta-1}E_{\theta, \theta}(-\lambda^{\theta}t^\theta), \quad F_{\theta, \lambda}(t) = 1 - E_{\theta}(-\lambda^\theta t^{\theta}).
\]
The corresponding counting process is an FPP of index $\theta$ and scale $\lambda$. We denote it by $N^{\theta, \lambda}$ and write 
\[
N^{\theta, \lambda} \sim \text{FPP}(\theta, \lambda).
\]
Similarly, $N^{1,\lambda}$ is the usual Poisson process at rate $\lambda$. Utilising this notation we also have 
\[
N^{\theta,1}(t) = N^{\theta}(t).
\]

Moreover, from equations \eqref{eq:renewaldef} and \eqref{eq:scaledML} we have the equality of processes 
\be \label{eq:lamove}
N^{\theta, \lambda}(t) = N^{\theta}(\lambda t), \qquad t \ge 0. 
\ee

The probability generating function for the distribution \eqref{eq:fpd} is, for $\theta \in (0,1]$ and $s \in [0,1]$,
\begin{equation}\label{eq:probgenfunction}
G_\theta (s,t) = \mathbb{E}[s^{N^\theta(t)}] = E_\theta (-t^\theta (1-s));
\end{equation}
see \cite[Eq.~(35)]{Laskin2003}. Note that when $\theta=1$ this formula recovers the probability generating function of the standard Poisson process.  Moreover, equality \eqref{eq:probgenfunction} holds true for all $s \in \mathbb C$. To see this, note that the r.h.s.\ is an entire function of $s$ for fixed $t$ and $\theta$ because so is $E_{\theta, \zeta}(s)$; see \cite{Wiman1905Fund} for when $E_{\theta, \zeta}$ was introduced and used as an entire function, as well as \cite[p.~118]{erde} for a book reference. Then $E_\theta (-t^\theta (1-s))$ equals its Maclaurin series (in $s$), which has an infinite radius of convergence. This series is $\mathbb{E}[s^{N^\theta(t)}] = \sum_{n=0}^\infty \mathbb{P}(N^\theta(t)=n) s^n/n!$ for $s \in [0,1]$ by \eqref{eq:probgenfunction}, and this equality extends to all $s \in \mathbb C$ by the change of variables formula for the Lebesgue integral since the series converges absolutely.

Since $G_{\theta}(s,t)$ is an entire function, from \eqref{eq:probgenfunction} we get a formula for the Laplace transform by setting $s = e^{\eta}$:
\begin{equation}\label{eq:mgf}
M_\theta (\eta,t)= \mathbb{E}[e^{\eta N^\theta(t)}] = E_\theta (-t^\theta (1-e^\eta)), \quad \eta \in \R.
\end{equation} 

This formula will be used for the moment generating function of the (light-tailed) inverse stable subordinator; see \eqref{eq:subMGF}. Note that taking $\eta \in i \R$ above gives a formula for the characteristic function of $N^\theta(t)$.

Equation \eqref{eq:lamove} is sufficient to derive the following equalities: 
\be\label{eq:scaledpgf}
G_{\theta,\lambda}(s,t) = \E( s^{N^{\theta,\lambda}(t)})=  \E( s^{N^{\theta}(\lambda t)})= E_{\theta}(\lambda^\theta t^{\theta}(s-1)), \quad s\in \R,
\ee
and 
\be\label{eq:scaledmgf}
M_{\theta,\lambda}(\eta,t) = \E( e^{\eta N^{\theta,\lambda}(t)})=  \E( e^{\eta N^{\theta}(\lambda t)})= E_{\theta}(\lambda^\theta t^{\theta}(e^\eta-1)), \quad \eta \in \R.
\ee

\subsection{Time change}
\label{subsec:tc}
Let $L_\theta(s)$ be a standard $\theta$-stable subordinator for a $\theta \in (0,1]$, that is a non-decreasing non-negative L\'evy process with the Laplace transform  
\[\E (e^{-t L_\theta(s)})=e^{-s t^\theta} \text{ for } s, t\ge 0.\]
Recall that the increments of a L\'evy process are stationary and independent. 
The case $\theta =1$ is special, since $L_1(s)=s$ is purely deterministic. 

Furthermore, define the process
\be \label{eq:subordinator}
Y_\theta(t)=\inf\{s>0: L_\theta(s)>t\}, \qquad t \ge 0,
\ee
called the \emph{inverse $\theta$-stable subordinator}. Such inverse is usually referred to as the right-continuous one, but since a.s.\ (almost surely, i.e.\ with probability $1$) there is no interval on which $L_\theta$ is  constant, it follows that the trajectories of $Y_\theta$ are continuous a.s. This inverse process is self-similar with index $\theta$, that is for every $s\ge 0$, there is the following distributional identity for the processes:
\be \label{eq:Yselfsimilarity}
\{Y_\theta(st)\}_{t \ge 0} \stackrel{d}{=}s^\theta\{Y_\theta(t)\}_{t \ge 0}.
\ee 
The random variable $Y_\theta(1)$ has density when $\theta \in (0,1)$, which we denote by $h_{\theta}(x)$; by the self-similarity,
$Y_\theta(t)$ has density $t^{-\theta} h_\theta(x t^{-\theta})$ when $t >0$. For $\theta=1$, we have 
 $Y_1(t)=t$ for $t \ge 0$.
 
The inverse stable subordinator offers a very convenient distributional representation of the fractional Poisson process, namely
\be \label{eq:timechange}
\{N^{\theta}(t)\}_{t \ge 0} \stackrel{d}{=} \{N^1(Y_\theta(t)) \}_{t \ge 0},
\ee
where $N^1$ is the rate $1$ standard Poisson process that is independent of $Y_\theta(t)$. As such, an FPP$(\theta)$ is a time-changed standard Poisson process, where the change is given by the inverse $\theta$-stable subordinator. 

Then, from \eqref{eq:mgf} and  \eqref{eq:timechange} we can compute
\begin{align*}
E_\theta (t^\theta (e^s-1))&=\mathbb{E}[e^{sN^\theta(t)}]=\mathbb{E}[e^{sN^1(t^\theta Y_{\theta}(1))}]\\
&= \int_0^{\infty}E[e^{sN^1(t^\theta x)}| Y_{\theta}(1)=x ]h_{\theta}(x)\, dx
= \int_0^{\infty}\E[e^{sN^1(t^\theta x)}]h_{\theta}(x)\, dx \\
&= \int_0^{\infty} e^{xt^{\theta}(e^s-1)}h_{\theta}(x)\, dx = \E(e^{t^{\theta} Y_{\theta}(1)(e^s-1)}) =  \E(e^{ Y_{\theta}(t)(e^s-1)}).
\end{align*}
In particular, changing variables $u = e^s - 1$, we have that 
\be\label{eq:subMGF}
\mathbb{E}(e^{uY_{\theta}(t)}) = E_{\theta}(ut^{\theta}), \qquad u \in \R,
\ee
as well as the functional relation 
\be \label{eq:mgfsequality}
\E(e^{ Y_{\theta}(t)(e^s-1)}) = \mathbb{E}(e^{sN^\theta(t)}).
\ee

Using \eqref{eq:lamove}, \eqref{eq:timechange}, \eqref{eq:Yselfsimilarity} and again \eqref{eq:lamove} in order, 
we obtain the distributional identities for the process $N^{\theta, \lambda} \sim \text{FPP}(\theta, \lambda)$ below:
\be\label{eq:EquivReps} 
N^{\theta, \lambda}(t) \stackrel{d}{=} N^{\theta}(\lambda t) \stackrel{d}{=} 
N^{1}(Y_\theta(\lambda t)) \stackrel{d}{=}
N^{1}(\lambda^\theta Y_{\theta}(t)) \stackrel{d}{=}
N^{1, \lambda^\theta}(Y_\theta(t)). 
\ee
The last identity above is also for notational convenience, as it pushes the rate of the Poisson process from a time scale in the argument into the second index above.  

\subsection{The discrete multi-class queueing model}
\label{subsec:model}
{We} introduce a {\it multiclass queue} for which arrivals and departures are governed by two independent fractional Poisson processes. 

As mentioned in \cite{ButtGeorgiouScalas2022}, while there are several equivalent descriptions of queueing systems when arrivals and departures are Markovian, the corresponding fractional queues exhibit different behaviors depending on the description. In this article, we are only considering {a single} particular model for arrivals and departures. Our model has the same arrival and departure structure as Model 3 {in} \cite{ButtGeorgiouScalas2022}, but we {endow} it with a finite set of priority classes. {Later, in Section 5,} we extend these ideas in a simple example with a continuum set of classes.

Let us describe the model. Customers arrive and enter the queue at the jump times of a fractional Poisson process $N^{\alpha, \lambda}(t) \sim \text{FPP}(\alpha, \lambda)$, where $\alpha \in (0,1]$. At time $t = 0$ the queue is empty and $N^{\alpha, \lambda}(0)=0$. Upon arrival, each customer, independently of everything else, is placed into one of $K$ priority classes $\{ C_1,  C_2, \ldots, C_K\}$. The customer is placed into Class $i$ with probability $p_i>0$, where $p_1, \ldots , p_K$ are fixed numbers such that $p_1 + \ldots +p_K=1$. This customer will get service priority over all customers of Classes $i+1, \ldots, K$, and will stand in line in front of all of them. However, the customer will never appear in front of anyone from Classes $1, \ldots, i$ who joined the queue earlier. Moreover, everyone joining Classes $1, \ldots, i-1$ before departure of the customer will stand in line in front of this customer. Thus, the classes follow the order of priority.  

The departures from this queue are as follows. Consider a fractional Poisson process  $D^{\beta, \mu}(t) \sim$ FPP$(\beta, \mu)$, where $\beta \in (0,1]$, independent of $N^{\alpha, \lambda}$. At each jump time of $D^{\beta, \mu}$, the earliest customer from the highest non-empty priority class exits the system unless the queue is empty, in which case nothing happens. 

It is important to clarify that not all jump times of $D^{\beta, \mu}$ are `service times'. On one hand, the time runs even if the queue is empty. On the other hand, if the queue is not empty, it could be that during an inter-event time of $D^{\beta, \mu}$, the process $N^{\alpha, \lambda}$ has a few jumps, and some of the new customers arriving can acquire higher priority than all of the customers who were previously in the queue. So at the next jump time of $D^{\beta, \mu}$, one of these new customers will exit instead of someone waiting in the system when the `service period' has started.  

Note that we can assume that only a single class of customers exists, i.e.\ $K = 1$. In fact, all of our results hold true in this simple case, and they either reduce to or generalise all the results for Model 3 in~\cite{ButtGeorgiouScalas2022}. 

Consider the total queue length $Q(t)$ at time $t$, without specifying the customer classes. It is given by 
\be \label{eq:fullqueuelength}
Q(t) = N^{\alpha, \lambda}(t) - D^{\beta, \mu}(t) - \inf_{s \le t}\{ N^{\alpha, \lambda}(s) - D^{\beta, \mu}(s)\}.
\ee
Equation \eqref{eq:fullqueuelength} is an instance of the Skorokhod reflection map $\Phi$ from the space $D[0, \infty)$ of c\`adl\`ag functions on $[0, \infty)$ into itself (see p.\ 87 in~\cite{Whitt2002}). 
It is defined as follows: for any such function $f$, put
\[
\Phi(f)(t) = f(t) +\sup_{0\le s\le t}\max\{-f(s),0\} = f(t) -\inf_{0\le s\le t}\min\{f(s),0\}, \qquad t \ge 0.
\]
Applying $\Phi$ on functions $f$ where $f(0)=0$, we see that the minimum does not affect the expression above, so in this case $\Phi$ can be simplified to 
\be \label{eq:Phi}
\Phi(f)(t) = f(t) -\inf_{0\le s\le t} f(s), \quad f(0)=0.
\ee
Therefore, since both arrival and departure processes start at 0, we see that sample path by sample path,
\be \label{eq:QasPhi}
Q = \Phi( N^{\alpha, \lambda} - D^{\beta, \mu}).
\ee

Each class $C_i$ has its own queue length at time $t$, which we denote by $Q_i(t)$. The total queue length  $Q(t)$ satisfies
\be \label{eq:Qparts}
Q(t) = \sum_{i=1}^K Q_i(t).
\ee
Moreover, we define the aggregated queue length of the first $i$ classes by 
\be\label{eq:Partiallength}
Q_{\le i}(t) = \sum_{\ell=1}^i Q_\ell(t)
\ee
Denote by $N^{(i)}(t)$ the arrival process of Class $i$, counting the number of arrivals into this class by time $t$. In other words, $N^{(i)}(t)$ increases by $1$ at time $\tau$ if and only if $N^{\alpha, \lambda}(t)$ increases by $1$ at time $\tau$ and the customer arriving is assigned to be in Class $i$. We denote the vector of the arrival processes by 
\be\label{eq:bfNdef}
{\bf N}(t) = \big(N^{(1)}(t), \ldots, N^{(K)}(t) \big).
\ee
Then we can write the full arrival process as the sum
\be\label{eq:arrivalsum}
N^{\alpha, \lambda}(t) = \sum_{\ell=1}^K N^{(\ell)}(t).
\ee
Similarly to \eqref{eq:Partiallength}, the partial sum of the first $i$ counting processes is denoted by 
\be\label{eq:Nipartials}
N^{\alpha, \lambda}_{\le i}(t) = \sum_{\ell=1}^i N^{(\ell)}(t).
\ee
Finally, define the cumulative probability mass function $P_0 =0$ and
\be \label{eq:Pi}
P_i =\sum_{j=1}^i p_j.
\ee

\section{Results}
\label{sec:results}

We are now ready to present the 
results for these models. They come in two types: (i) distribution and scaling limits for thinned FPP processes and (ii) scaling limits and recurrence for the multiclass queue lengths. The application to the non-discrete multiclass model is in its own Section \ref{sec:doubleauction}.  

Throughout Section 2 we assume that $\alpha, \beta \in (0,1]$ and $\lambda, \mu>0$. 
Taking $\alpha=1$ (resp.\ $\beta=1$) recovers the classical Poisson arrivals (resp.\ services) as a special case.

\subsection{The thinned FPP} 
The first thing to show is a characterisation for the vector of arrival processes. It shows that the classification of customers still maintains an aspect of independence inherited from the classical Poisson thinning, and all correlations enter via the time change with the \emph{same} random clock.

\begin{theorem}\label{thm:babyarrival}
Let ${\bf N}(t)$ be as in \eqref{eq:bfNdef}. Then we have a distributional equality of processes 
\[
\{{\bf N}(t)\}_{t \ge 0}  \stackrel{d}{=} \big\{\big(N_{1}(p_{1} Y_\alpha(t)), \ldots, N_{K}(p_{K} Y_{\alpha}(t)) \big) \big\}_{t \ge 0},
\]
where $N_{1}(t), \ldots, N_{K}(t)$ are Poisson processes of rate $\lambda^{\alpha}$, $Y_\alpha(t)$ is the inverse of a standard $\alpha$-stable subordinator, and these $K+1$ processes are mutually independent.
\end{theorem}

\begin{remark} 
Theorem \ref{thm:babyarrival} says that if we mark the events of an FPP independently into $K$ types, the vector of the marked counting is distributionally equivalent to a vector of independent Poisson coordinates time-changed by the same random clock $Y_\alpha$. 

As it turns out, the idea of time changing a vector of independent  Poisson processes with a common random clock exists in the literature. This is exactly the \emph{multivariate space–time fractional Poisson construction} of \cite{BeghinMacci2016}. The novelty in our paper is that the multivariate space-time FPP naturally arises as a result of the marking procedure, rather than being defined directly as the model under study.  
\end{remark}

\begin{remark}
For $\alpha=1$, our randomized class assignment at arrivals corresponds to the well-known procedure of thinning of a standard Poisson process $N^1$, which results in representing $N^1$ as a sum \eqref{eq:arrivalsum} of independent Poisson processes. Naturally, for $\alpha \in (0,1)$ these processes are no longer dependent. Corollary \ref{cor:babyarrival} below highlights how their correlations 
enter the picture.  
\end{remark}

\begin{corollary}\label{cor:babyarrival}
Let ${\bf N}(t)$ be as in \eqref{eq:bfNdef}. The Laplace transform  of ${\bf N}(t)$ is
\be \label{eq:NImgf}
\E e^{s \cdot {\bf N}(t)} = E_{\alpha}\Big( \lambda^\alpha t^{\alpha} \sum_{i=1}^K p_{i}(e^{s_{i}} - 1)\Big), 
\ee
where $s= (s_1, \ldots, s_K) \in \R^K$ and $t \ge 0$. The class arrival processes $N^{(i)}$ are distributed as 
\be \label{eq:1dimFPP0}
\{N^{(i)}(t)\}_{t \ge 0} \stackrel{d}{=} \{ N^{\alpha, \lambda p_i^{1/\alpha}}(t)\}_{t \ge 0} \sim {\rm FPP}(\alpha, \lambda p_i^{1/\alpha}).
\ee
Moreover, for any $\ell \le K$,
their sum \eqref{eq:Nipartials} up to $\ell$ satisfies 
\be\label{eq:ThinSum0} 
\Big\{N_{\le \ell}^{\alpha, \lambda}(t)\Big\}_{t\ge 0} \sim {\rm FPP}\big(\alpha, \lambda P_\ell^{1/\alpha}\big).
\ee
Finally, for any $t \ge 0$, the covariance between two coordinates $i \neq j$ is given by 
\begin{align} \label{eq:ThinCov0}
{\rm Cov}(N^{(i)}(t),N^{(j)}(t)) &= p_ip_j\lambda^{2\alpha} t^{2\alpha}\left( \frac{1}{\Gamma(1+2\alpha)} - \frac{1}{(\Gamma(1+\alpha))^2}\right)\\
&= p_ip_j\lambda^{2\alpha}{\rm Var}(Y_{\alpha}(t)).
\end{align}
\end{corollary}

\begin{remark}
Given a fixed total number of customers $N^{\alpha, \lambda}(t)$, the class counts “compete’’, in the sense that their conditional covariance is negative. In fact, it follows from~\eqref{eq:multi:pmf1} below that
\begin{align*}
{\rm Cov}(N^{(i)}(t), N^{(j)}(t))| N^{\alpha, \lambda}(t)) 
&= \E(N^{(i)}(t)N^{(j)}(t)|N^{\alpha, \lambda}(t)) \\
&\phantom{xxxxxxxxxxxx}- \E(N^{(i)}(t)|N^{\alpha, \lambda}(t))\E(N^{(j)}(t)|N^{\alpha, \lambda}(t))\\
&=-p_ip_j N^{\alpha, \lambda}(t)\le 0.
\end{align*}
However, it is readily seen from \eqref{eq:ThinCov0} that the total covariance between the coordinates is positive. This is an artifact of the total itself being random through a common inverse–stable clock. Busy periods speed up \emph{all} class streams simultaneously and quiet periods slow them all down.
This shared random environment creates \emph{positive} co–movement across classes, vanishing only in the Poisson case ($\alpha=1$) where the clock is deterministic.
Practically, positive cross–covariance means pooling is less variance–reducing than in classical models and extremes tend to hit all classes together. This observation is crucial for staffing and tail risk.
\end{remark}

Next, we are interested in distributional properties of ${\bf N}(t), Q(t)$, and $(Q_1(t), \ldots, Q_K(t))$. The first two theorems are about process-level convergence for ${\bf N}(t)$ in the form of an LLN and a FCLT. 

\begin{theorem}\label{thm:LLN}
Let ${\bf N}(t)$ be as in \eqref{eq:bfNdef}, and let ${\bf p} = (p_1, \ldots, p_K)$. We have the weak convergence 
\[
\big\{u^{-\alpha} {\bf N}(ut) \big\}_{t \ge 0} \stackrel{d}{\longrightarrow} \{ \lambda^{\alpha} {\bf p} Y_\alpha(t)\}_{t \ge 0}, \qquad u \to \infty,
\]
in the space $D([0,\infty),\R^K)$ equipped with the Skorokhod topology $J_1$, where $Y_\alpha$ is the inverse of a standard $\alpha$-stable subordinator. 
\end{theorem}

We close this subsection with a functional central limit theorem for the process ${\bf N}(t)$. We use its time-changed representation $(N_1(p_1 Y_{\alpha}(t)), \ldots, N_K(p_K Y_{\alpha}(t)))$, allowed by Theorem \ref{thm:babyarrival}, and we work directly with this process. For each customer class $i=1,\dots ,K$, define the centred and normalised process
\[
   M_u^{(i)}(t)
   \;=\;
   \frac{N_i(p_i Y_{\alpha}(u t))-p_i \lambda^{\alpha} Y_\alpha(u t)}{u^{\alpha/2}},
   \qquad t \ge 0,\;u>0 .
\]

\begin{theorem}[Functional CLT for thinned FPP]\label{thm:fclt}
Let $B_1,\dots ,B_K$ be independent standard Brownian motions, independent of $Y_\alpha$. Then we have the weak convergence
\[
   (M_u^{(1)},\dots ,M_u^{(K)}) \stackrel{d}{\longrightarrow} \big \{\bigl(B_1(p_1 \lambda^{\alpha}\,Y_\alpha(t)),\dots ,
         B_K(p_K\lambda^{\alpha}\,Y_\alpha(t))\bigr) \big \}_{t \ge 0},
\]
in the space $D\!\bigl([0,\infty),\mathbb R^{K}\bigr)$ equipped with the Skorokhod topology $J_1$, as $u \to \infty$.
\end{theorem}

\begin{remark}This theorem applies directly to the multivariate space-time fractional Poisson model of \cite{BeghinMacci2016}. 
\end{remark}

\subsection{The discrete-class multiclass queue: Scaling limits and recurrence}

\begin{theorem}[Scalings of queue lengths]\label{thm:QConvergence}
Let $Y_{\alpha}(t)$ and $\widetilde Y_{\alpha}(t)$ be two independent inverse standard $\alpha$-stable subordinators, and let $\gamma = \max\{\alpha, \beta\}$.
Then for any $i \le K$, for $Q_{\le i}$ given by \eqref{eq:Partiallength}, $P_i$ given by~\eqref{eq:Pi}, and $\Phi$ given by~\eqref{eq:Phi}, we have the weak convergence
\begin{align}
\left\{ \frac{Q_{\le i}(u t)}{u^\gamma} \right \}_{t \ge 0}&\stackrel{d}{\longrightarrow} \displaystyle 
\begin{cases} 
\displaystyle \lambda^\alpha P_i Y_{\alpha}, & \quad \gamma = \alpha >\beta,\\
0, &\quad \gamma = \beta > \alpha, \\ 
\displaystyle  \Phi\big(\lambda^\alpha P_i Y_\alpha - \mu^\alpha \widetilde Y_\alpha \big), &\quad \gamma = \alpha = \beta,
\end{cases}
\label{eq:Qilimit}
\end{align}
in the space $D[0, \infty)$ equipped with the Skorokhod topology $J_1$, as $u \to \infty$ (above, $0$ is as a function). 
Moreover, in the case where $\alpha = \beta$, the individual class queue lengths satisfy 
\begin{align}
\left\{ \frac{Q_i(u t)}{u^\gamma} \right \}_{t \ge 0} &\stackrel{d}{\longrightarrow}
\Phi\big(\lambda^\alpha P_i Y_\alpha - \mu^\alpha \widetilde Y_\alpha\big)-\Phi\big(\lambda^\alpha P_{i-1} Y_\alpha - \mu^\alpha \widetilde Y_\alpha\big).
\label{eq:individualq}
\end{align}
\end{theorem}

Having obtained a first-order scaling limit in Theorem \ref{thm:QConvergence} and the functional CLT in Theorem \ref{thm:fclt} for the class counts, the next natural step is to combine the two to obtain a functional limit theorem for the queue lengths $Q_{\le i}$, centered appropriately.   

Assume, as in Theorem \ref{thm:fclt}, that the class arrival process is given by $(N_1(p_1 Y_{\alpha}(t)), \ldots, N_K(p_K Y_{\alpha}(t)))$, and the service departure process is $\widetilde{N}(\widetilde{Y}_{\beta}(t))$, where $N_1, \ldots, N_K$ are Poisson processes of rate $\lambda^{\alpha}$  and $\widetilde{N}$ is a Poisson process of rate $\mu^{\beta}$, 
$Y_\alpha$ and $\widetilde Y_\beta$ are inverse standard stable subordinators of respective indices $\alpha$ and $\beta$, and all the $K+3$ processes are independent. For fixed $i, \lambda, \mu, p_1, \ldots, p_i$, which will not appear in the shorthanded notation below, define 
\be\label{eq:2diff}
N_{\alpha, \beta}^{\rm diff}(t) = \sum_{j=1}^i N_j( p_j Y_\alpha(t)) - \widetilde{N}(  \widetilde{Y}_{\beta}(t)), \quad Y_{\alpha, \beta}^{\rm diff}(t) =\lambda^\alpha P_i Y_\alpha(t) - \mu^\beta \widetilde Y_\beta(t).
\ee
Then we define $Q^{\rm cent}_{\le i}(t)$, the ``centered''  aggregated queue length of Classes 1 through $i$, to be 
\be\label{Not1}
Q^{\rm cent}_{\le i} = \Phi \big( N_{\alpha, \beta}^{\rm diff} - Y_{\alpha, \beta}^{\rm diff} \big),
\ee
i.e.~ the Skorokhod reflection map applied to the difference of the centered arrival and departure processes. Then we have the following Corollary of Theorem \ref{thm:fclt}.  

\begin{corollary}[A functional CLT for a centered queue] 
\label{cor:qclt}
Let $X_{\alpha}$ and $\widetilde X_{\beta}$ be  inverse standard stable subordinators of respective indices $\alpha$ and $\beta$, $B$ and $\widetilde B$ be standard Brownian motions, and all four processes be independent. 

Then for $Q^{\rm cent}_{\le i}$ in \eqref{Not1}, we have the weak convergence
\begin{align*}
\left\{ \frac{Q^{\rm cent}_{\le i}(ut)}{u^{\gamma/2}} \right\}_{t \ge 0}
\stackrel{d}{\longrightarrow}\begin{cases}
\Phi\left(\sqrt{\lambda^\alpha P_i} B \circ X_\alpha\right), &\quad \gamma = \alpha > \beta,\\ 
\Phi\big(\sqrt{\lambda^\alpha P_i} B \circ X_\alpha-\sqrt{\mu^\beta} {\widetilde{B}}\circ \widetilde X_\alpha\big), &\quad \gamma = \alpha = \beta, \\
  \Phi\big(\sqrt{\mu^\beta} \widetilde B \circ \widetilde X_\beta\big), &\quad \gamma = \beta > \alpha.
\end{cases}
\end{align*}
in the space $D\!\bigl([0,\infty),\mathbb R^{K}\bigr)$ equipped with the Skorokhod topology $J_1$, as $u \to \infty$.
\end{corollary}

\begin{remark} A different centered queue length can be defined by
\[
\widetilde{Q}^{\rm cent}_{\le i} = \Phi \big( N_{\alpha, \beta}^{\rm diff}\big) - \Phi \big(Y_{\alpha, \beta}^{\rm diff} \big).
\]
This would be in the spirit of classical scaling limits such as CLT where to get a weak convergence we center the random process around its law of large numbers. Let 
\be \label{Not2}
H_{\alpha, \beta}(t)={Q}^{\rm cent}_{\le i}(t)-\widetilde{Q}^{\rm cent}_{\le i}(t). 
\ee
We believe that $H_{\alpha, \beta}(u\cdot)/u^{\gamma/2}$ has a non-trivial limit, and as such we should be able to obtain scaling limit for $\widetilde{Q}^{\rm cent}_{\le i}$.  
This study will be left as part of a future work.    
\end{remark}

Finally, our last theorem concerns `recurrence' in the sense of the queue emptying infinitely often with probability 1, and `transience' in the sense that with probability 1 there is a subsequence of random times $t_n \to \infty$ as $n \to \infty$, such that $Q(t_n)\to \infty$. We only state the result for the case $\alpha = \beta$.   

In \cite{ButtGeorgiouScalas2022}, the proof methodology utilised certain couplings to argue that one should expect the restless Mittag-Leffler queue (fractional Model 3) to empty infinitely often when  $\alpha = \beta$ and the dilation parameters $\lambda$ and $\mu$ satisfied the conditions $\mu \ge \lambda$. This condition for recurrence of the queue appears to be natural since it is required in the classical light-tailed M/M/1 setting. However, it turns out that in our heavy-tailed setting this condition is excessive, as shown in the following result.

\begin{theorem}[Recurrence and Transience]\label{thm:Rec+Trans}
In the critical case $\alpha=\beta$ in Theorem~\ref{thm:QConvergence}, the following statements hold true, irrespective of the values of $\lambda, \mu, p_i$.
\begin{enumerate}
\item The queue empties infinitely often with probability $1$.
\item We have
$
\varlimsup_{t\to \infty} Q_{\le i}(t) = +\infty 
$
with probability 1.
\end{enumerate}
\end{theorem}

\begin{remark} When $\alpha < \beta$, the (total) arrival process is slower than the departure process, so the queue will empty infinitely often with probability 1. This is Theorem 3.4 in \cite{ButtGeorgiouScalas2022} (see also Theorem \ref{thm:3435} in Appendix \ref{app:A}). The same theorem gives that if $\alpha > \beta$, then $\limsup$ of the queue length is $+\infty$ with probability 1. Moreover, by Theorem~3.3 in \cite{ButtGeorgiouScalas2022} (Theorem \ref{thm:33} in Appendix \ref{app:A}), the scaling limit of the queue length would be distributionally equal to the inverse $\alpha$-stable subordinator. 
Here, both theorems can be applied to the queue length of Class 1 since the arrival process for it is FPP by Theorem \ref{thm:babyarrival}, to show corresponding results.  

For completion we record the statements, but it is a direct application of Theorem 3.4. in \cite{ButtGeorgiouScalas2022} (first part of Theorem \ref{thm:3435} in Appendix \ref{app:A} so the proof will be omitted. 
\begin{theorem}
The following statements hold true, irrespective of the values of $\lambda, \mu, p_i$. For any fixed $i \le K$,
\begin{enumerate}
	\item If $\alpha < \beta$, the queue empties infinitely often with probability $1$. 
	\item If $\alpha > \beta$, we have
$
\varlimsup_{t\to \infty} Q_{\le i}(t) = +\infty 
$
with probability 1. 
\end{enumerate} 
\end{theorem} 
\end{remark}


\section{Proofs for the arrival process}
\label{sec:thinproof}

Denote by $\tau_k$ the instant 
of $k$-th jump of the FPP $N^{\alpha, \lambda}$ of arrivals. We have $0<\tau_1< \tau_2< \ldots$ a.s. Recall that $\bf N$ is the multivariate process of class counts, defined in \eqref{eq:bfNdef}. This process is piecewise constant. Its jumps, given by 
\begin{equation} \label{eq: jumps in R^K}
{\W}_k={\bf N}(\tau_k)- {\bf N}(\tau_k-),
\end{equation}
occur at the moments $\tau_k$. 
Our procedure of independent assignment of customers into classes means exactly that $\{\W_k\}_{k \ge 1}$ are i.i.d.\ random vectors with the common distribution 
\[
\P\{ \W_k = e_i\} = p_i,
\]
where $e_1, \ldots, e_K$ are the standard orthonormal basis vectors of $\R^K$. This is because the position of the only non-zero coordinate of $\W_k$  indicates the class of the $k$th customer arriving into the queue. Moreover, by our assumption of independent assignments, the sequence $\{\W_k\}_{k \ge 1}$ is independent of the FPP $N^{\alpha, \lambda}$.

Consider a random walk on $\R^K$:
\[
\fS(n) = \sum_{i=1}^n \W_i, \quad \fS(0) = 0.   
\]
For each $n$, its position $\fS(n)$ has the multinomial distribution with parameters $(n, K, p_1, \ldots, p_K)$. The increments of this walk assign  classes of customers in our multiclass queue. Namely, it follows from~\eqref{eq: jumps in R^K} that
\begin{equation} \label{eq: NN=S(N)}
{\bf N}(t)= \fS(N^{\alpha, \lambda}(t)), \qquad t \ge 0,
\end{equation}
where $\fS$ is independent of $N^{\alpha, \lambda}$. In particular, given $N^{\alpha, \lambda}(t)$, the class count ${\bf N}(t)$ has a multinomial distribution:
\begin{align}
		\P\{ {\bf N}(t)&= (n_1,\ldots, n_K) | N^{\alpha, \lambda}(t) = n\} \notag \\
		&=\begin{cases} 
						\displaystyle \genfrac{(}{)}{0pt}{}{n}{n_1,\ldots, n_K} p_1^{n_1} \ldots p_K^{n_K}, & \text{ if } n_1+ \ldots + n_K =n,
						\\
						0, & \text{ otherwise. }
		\label{eq:multi:pmf1}
				\end{cases}
\end{align}

\begin{remark}
A different multivariate model with the same conditional multinomial law given the sum as in \eqref{eq:multi:pmf1} can be found in~\cite{BeghinMacci2017}. However in that case the Poisson process was not time-changed by the inverse stable subordinator. In fact, the approach presented here -- and also the proof of Theorem \ref{thm:babyarrival} -- does not use the form of the distribution of the time-change, and therefore it can be used by conditioning on arbitrary counting processes. 
\end{remark}

Let us use representation~\eqref{eq: NN=S(N)} to prove Theorem \ref{thm:babyarrival}.

\begin{proof}[Proof of Theorem \ref{thm:babyarrival}]
We need to establish equalities of all finite-dimensional distributions of these $\R^K$-valued random processes. For any fixed $n \in \N$ and time instants 
$0 \le t_1 \le t_2\le \ldots\le  t_n$ we will prove a distributional equality conditional on 
\[ (Y_{\alpha}(t_1), \ldots Y_{\alpha}(t_n)) = (y_1, y_2, \ldots, y_n), \]
for $0 \le y_1 \le \ldots \le y_n$,
and then obtain the corresponding equality of $n$-dimensional distributions by integrating over the distribution of $ (Y_{\alpha}(t_1), \ldots Y_{\alpha}(t_n))$.

Using \eqref{eq:EquivReps} and \eqref{eq: NN=S(N)}, we see that the process ${\bf N}(t)$ has the same distribution as $\fS(N(Y_\alpha(t)))$, where $N$ is a shorthand for $N^{1, \lambda^\alpha}$. We now condition on $Y_\alpha$. It then suffices to show that
for any $n \in \N$ and  $0 \le y_1 \le \ldots \le y_n$,
\begin{align} \label{eq: classical thinning}
\big(\fS(N(y_1)), \ldots, &\fS(N(y_n)) \big) \notag \\
&\stackrel{d}{=} \Big(\big(N_{1}(p_1 y_1), \ldots, N_K(p_K y_1) \big), \ldots, \big(N_1(p_1 y_n), \ldots, N_K(p_K y_n) \big) \Big). 
\end{align}
This is equivalent
\footnote{Equalities~\eqref{eq: classical thinning} all together mean that
$
\{\fS(N(y)\}_{y \ge 0} \stackrel{d}{=} \big\{\big(N_1(p_1 y), \ldots, N_K(p_K y) \big) \big\}_{y \ge 0},
$
which describes the classical thinning of the Poisson processes $N$ using $K$ marks with probabilities $p_i$. Indeed, the l.h.s.\ is simply the multivariate process of class counts in our queue if we assume that the process of arrivals is $N$. 
The marginals of $\fS(N(y))$ are mutually independent Poisson processes of rate $p_i \lambda^\alpha$, according to the thinning. However, we prefer to give a full proof of ~\eqref{eq: classical thinning}. 
} 
to establishing the following joint distributional identity for the increments:
\begin{align*}
\quad \,\, \big(\fS(N^{1, \lambda^\alpha}(y_1)),  \, &\fS(N(y_2))-\fS(N(y_1)), \ldots, \fS(N(y_n))-\fS(N(y_{n-1})) \big) \\
&\stackrel{d}{=} \Big(\big(N_1(p_1 y_1), \ldots, N_K(p_K y_1) \big), \ldots ,\\
&\phantom{xxxxxxxx} \big(N_1(p_1 y_n)- N_1(p_1 y_{n-1}), \ldots, N_K(p_K y_n)-N_K(p_K y_{n-1}) \big) \Big).
\end{align*}
Since the random walk $\fS$ has independent increments, by conditioning on $N$ we see that the  increments on the l.h.s.\ (which are $n$ vectors in $\R^K$) are independent. The matching $n$ vectors on the r.h.s.\ are also independent (and moreover, have independent coordinates) because $N_1, \ldots, N_K$ are independent processes with independent increments. Therefore, it suffices to show the marginal equalities
\[
\fS(N(y_i))-\fS(N(y_{i-1}))
\stackrel{d}{=} \big(N_1(p_1 y_i)- N_1(p_1 y_{i-1}), \ldots, N_K(p_K y_i)-N_K(p_K y_{i-1}) \big)
\]
for all $1 \le i \le n$, where $y_0:=0$. 

Since the processes $\fS(N(t))$ and $N_i(t)$ have stationary increments, the above reduces to  
\be \label{eq:fullmarkov}
\fS(N(y))
\stackrel{d}{=} \big(N_1(p_1 y), \ldots, N_K(p_K y) \big)
\ee
for all $y \ge 0$. 

By independence of increments of $N$, this is in turn equivalent to 
\[
\fS(N(y)) \stackrel{d}{=} \big(N(p_1 y), N((p_1 +p_2) y) - N(p_1 y), \ldots, N( y) - N((p_1+\ldots +p_{K-1}) y) \big).
\]
This last equality follows from the fact that given $N(y)=m$, the instants of jumps of $N$ on $[0,y]$ are i.i.d.\ and uniformly distributed over $[0,y]$. In fact, the r.h.s.\ represents the numbers of these jumps in the intervals $(0, p_1y ], \ldots, ((p_1+ \ldots + p_{K-1}) y, y]$. The joint distribution of these numbers for all $K$ intervals is  the multinomial distribution with parameters $(m, K, p_1, \ldots, p_K)$, which is also the distribution of $\fS(m)$.
Therefore, \eqref{eq:fullmarkov} follows, and this suffices for the theorem. 
\end{proof}

\begin{proof}[Proof of Corollary \ref{cor:babyarrival}] 

The distributional identity of Theorem \ref{thm:babyarrival} can directly give the Laplace transform~\eqref{eq:NImgf} by conditioning on $Y_{\alpha}$ and  then referring to~\eqref{eq:subMGF}. However, even more basic means can be used, as we show below. For $s= (s_1, s_2, \ldots, s_K)$, by \eqref{eq:multi:pmf1} we have  
\begin{align*}
\E  \, e^{s \cdot {\bf N}(t)}
 &= \sum_{n=0}^\infty E\big( e^{s \cdot {\bf N}(t)}\big|N^{\alpha, \lambda}(t)=n \big)\P\{ N^{\alpha, \lambda}(t) = n\}\\
&=\sum_{n=0}^\infty \P\{ N^{\alpha, \lambda}(t) = n\} \times \sum_{n_1+ \ldots + n_K=n} e^{s \cdot (n_1, \ldots, n_{\ell})}  \genfrac{(}{)}{0pt}{}{n}{n_1,\ldots, n_K} p_1^{n_1} \ldots p_K^{n_K} \\
&=\sum_{n=0}^\infty \P\{ N^{\alpha, \lambda}(t) = n\} \times \sum_{n_1+ \ldots + n_K=n} \genfrac{(}{)}{0pt}{}{n}{n_1,\ldots, n_K} (e^{s_1} p_1)^{n_1} \ldots (e^{s_K} p_K)^{n_K} \\
&= \sum_{n=0}^\infty \P\{ N^{\alpha, \lambda}(t) = n\} \Big( \sum_{j=1}^K e^{s_j}p_j \Big)^n\\
&= \E\Bigg[\Big( \sum_{j=1}^K e^{s_j}p_j \Big)^{N^{\alpha, \lambda}(t)}\Bigg] = G_{\alpha, \lambda}\Big( \sum_{j=1}^K p_j e^{s_j}\Big).
\end{align*}
Using \eqref{eq:scaledpgf}, the last line becomes 
\[
\E  \, e^{s \cdot {\bf N}(t)} = E_{\alpha}\Big( \lambda^\alpha t^{\alpha} \sum_{j=1}^K p_j(e^{s_j} - 1)\Big),  
\]
thus proving \eqref{eq:NImgf}.

Same arguments as above (i.e. conditioning on the number of events and then using the properties of the multinomial distribution) give pairwise (or more) correlation functions. For example, for the covariance between coordinates, use the conditional covariance formula
\begin{align*} 
{\rm Cov}(N^{(i)}(t), N^{(j)}(t)) &=\E({\rm Cov}(N^{(i)}(t), N^{(j)}(t))| N^{\alpha, \lambda}(t)) \\
&\phantom{xxxxxxx}+ {\rm Cov}(E(N^{(i)}(t)|N^{\alpha, \lambda}(t)), E(N^{(j)}(t)|N^{\alpha, \lambda}(t)))\\
&= -p_ip_j\E(N^{\alpha, \lambda}(t)) + p_ip_j {\rm Var}(N^{\alpha, \lambda}(t))\\
&=p_ip_j \big [\E ({\rm Var}(N^{\alpha, \lambda}(t)|Y_{\alpha})) \\
&\phantom{xxxxxxx}+{\rm Var}(E(N^{\alpha, \lambda}(t)|Y_{\alpha}(t))) - \E(E( N^{\alpha, \lambda}(t)| Y_{\alpha}(t))) \big ]\\
&=p_ip_j \lambda^{2\alpha} {\rm Var}(Y_{\alpha}(t)).
\end{align*}
This gives one of the two equalities in \eqref{eq:ThinCov0}. The other equality follows from  \eqref{eq:subMGF} and the definition of the Mittag-Leffler function $E_\alpha$:
\[{\rm Var}(Y_{\alpha}(t)) = \big((E_\alpha( u t^\alpha))''/2 - [(E_\alpha( u t^\alpha))']^2  \big) \big . \big |_{u=0} =  t^{2\alpha}\left( \frac{1}{\Gamma(1+2\alpha)} - \frac{1}{(\Gamma(1+\alpha))^2}\right). \]

To obtain \eqref{eq:1dimFPP0}, use Theorem \ref{thm:babyarrival} and project to a single coordinate, and then apply  \eqref{eq:EquivReps}.

 Finally, let us prove \eqref{eq:ThinSum0}. From Theorem \ref{thm:babyarrival} and the fact that addition of the first $\ell$ coordinates is a measurable map from $\R^K$ to $\R$, we have a distributional equality of processes 
 \[
 \Big\{\sum_{j=1}^\ell N^{(j)}(t)\Big\}_{t\ge 0}  \stackrel{d}{=}  \Big\{\sum_{j=1}^\ell N_j(p_jY_{\alpha}(t))\Big\}_{t\ge 0}.
 \]
 Denote by $M(t)$ the random process on the r.h.s.\ above. To identify its distribution, use that the sum of $\ell$ independent Poisson processes $ N_j(p_j y)$ of rate $p_j \lambda^\alpha$, $1 \le j \le \ell$, is a Poisson process of rate $\lambda^\alpha P_\ell$.
 Then for any fixed $n \in \N$, integers $0\le k_1 \le \ldots \le k_n$, and reals $0 \le t_1 \le \ldots \le  t_n$ and $0 \le y_1 \le \ldots \le y_n$,
 \begin{align*}
 \P \Big ( M(t_1)=k_1, \ldots &, M(t_n)=k_m  \Big | \big. Y_{\alpha}(t_1) = y_1, \ldots, Y_{\alpha}(t_n) = y_n \Big ) \\
 &  = \P \Big ( N^{1, \lambda^\alpha P_\ell}(y_1) =k_1, \ldots, N^{1, \lambda^\alpha P_\ell}(y_n) =k_n \Big ).
 \end{align*}
Integrating this equality over the joint distribution of $(Y_{\alpha}(t_1), \ldots Y_{\alpha}(t_n))$, we arrive at
\[
\{M(t_i)\}_{1 \le i \le n} \stackrel{d}{=} \{ N^{1, \lambda^\alpha P_\ell}(Y_\alpha(t_i))\}_{1 \le i \le n},
\]
which in turn implies that
\[
\{M(t)\}_{t \ge 0} \stackrel{d}{=} \{ N^{1, \lambda^\alpha P_\ell}(Y_\alpha(t))\}_{t \ge 0}.
\]
By \eqref{eq:EquivReps}, the process on the r.h.s.\ is FPP$(\alpha, \lambda  P_\ell^{1/\alpha})$, as needed for \eqref{eq:ThinSum0}.
\end{proof}

\begin{proof}[Proof of Theorem \ref{thm:LLN}]
Since the limit process $\lambda {\bf p} Y_\alpha(t)$ has continuous trajectories, by \cite[Theorem~16.7]{billingsley1999}\footnote{All the results of~\cite{billingsley1999} cited here are stated for the space $D([0,T],\R)$, but they shall transfer directly to $D([0,T],\R^K)$.} it suffices to establish,  for every $T>0$,  the weak convergence $\{u^{-\alpha} {\bf N}(ut)\}_{t \in [0,T]} \stackrel{d}{\longrightarrow} \{ \lambda^\alpha {\bf p} Y_\alpha(t) \}_{t \in [0,T]}$ of these processes considered as elements of $D([0,T],\R^K)$. In view of Theorem~\ref{thm:babyarrival}, it suffices to prove this for $u^{-\alpha} {\bf N}(ut)$ replaced by
\[
\tilde {\bf N}_u(t) = u^{-\alpha} \big(N_1(p_1 Y_\alpha(ut)), \ldots, N_K(p_K Y_\alpha(ut)) \big).
\]

We first prove that  for every $T>0$,
\begin{equation} \label{eq:sup_diff}
\sup_{0 \le t \le T } \Big |\tilde {\bf N}_u(t) - \lambda^{\alpha} u^{-\alpha} {\bf p} Y_{\alpha}(ut) \Big | \stackrel{\P}{\longrightarrow} 0 
\end{equation}
as $u \to \infty$. To this end, for any $t>0$, denote 
\[
L(t)= \Big(\frac{N_1(p_1 Y_{\alpha}(t))}{ Y_\alpha(t)}, \ldots, \frac{N_K(p_K Y_{\alpha}(t))}{Y_\alpha(t)} \Big).
\]
This quantity is well-defined because $Y_\alpha(t)>0$ a.s.

We have $\tilde  {\bf N}_u(t)=u^{-\alpha} L(ut) Y_\alpha(ut)$. Then, since the processes $N_1, \ldots, N_K, Y_\alpha$ are non-decreasing, for any $u>1$,
\begin{align}
\sup_{0 \le t \le T } \Big |\tilde  {\bf N}_u(t) &- \lambda^{\alpha} u^{-\alpha} {\bf p} Y_\alpha(ut) \Big | \notag \\
&\le |\tilde  {\bf N}_u(1/\sqrt u) | + \lambda^\alpha u^{-\alpha} Y_\alpha (\sqrt u) |{\bf p}|  + \sup_{1/\sqrt u \le t \le T } \Big | \tilde  {\bf N}_u(t) - \lambda^\alpha u^{-\alpha} {\bf p} Y_\alpha(ut) \Big | \notag \\
&\le u^{-\alpha} Y_\alpha(\sqrt u) |L(\sqrt u)| + \lambda^\alpha u^{-\alpha} Y_\alpha (\sqrt u) |{\bf p}| \notag \\
&\phantom{xxxxxxxxxxxxxxxxxx}\!+ u^{-\alpha} Y_\alpha(uT) \sup_{1/\sqrt u \le t \le T} \big |L(ut) - \lambda^\alpha{\bf  p} \big |.  \label{eq:sup_diff2}
\end{align}
By the strong law of large numbers applied to the Poisson processes $N_1, \ldots, N_K$ and the fact that  $Y_\alpha(t) \to \infty$ a.s.\ as $t \to \infty$, we have $L(t) \to \lambda^\alpha {\bf p}$ a.s.\ as $t \to \infty$; to see this, divide and multiply by $p_i$ each coordinate $i$ of $L(t)$. Therefore, the last supremum in \eqref{eq:sup_diff2} vanishes a.s.\ as $u \to \infty$. Combined with the distributional equalities $u^{-\alpha} Y_\alpha(\sqrt u) \stackrel{d}{=} u^{-\alpha/2} Y_\alpha(1)$ and $u^{-\alpha} Y_\alpha(uT) \stackrel{d}{=} Y_\alpha(T)$, this implies that the r.h.s.\ of \eqref{eq:sup_diff2} converges to $0$ in probability, establishing \eqref{eq:sup_diff}. 

Furthermore, the standard Skorokhod metric $d$, which metrizes the topology $J_1$ on $D([0,T],\R^K)$, is shorter 
than the uniform metric on $[0, T]$; see Eq.~(12.13) in~\cite{billingsley1999}. Therefore, from \eqref{eq:sup_diff} we get
\begin{equation} \label{eq:Skorokhod in P}
d\big (\tilde {\bf N}_u, \lambda^{\alpha} u^{-\alpha}{\bf p} Y_\alpha(u \, \cdot) \big ) \stackrel{\P}{\longrightarrow} 0
\end{equation}
as $u \to \infty$. Let us clarify why $d\big (\tilde {\bf N}_u, \lambda^{\alpha} u^{-\alpha}{\bf p} Y_\alpha(u \, \cdot) \big )$ is a well-defined random variable. Indeed, we first claim that the pair $ (\tilde {\bf N}_u, \lambda^\alpha u^{-\alpha} {\bf p} Y_\alpha(u \, \cdot) \big )$ is a random element of the product space $D([0,T],\R^K) \times D([0,T],\R^K)$ equipped with $J_1 \times J_1$. This is because the Skorokhod space is separable by Theorem~12.2 in~\cite{billingsley1999}, and then the claim follows using the results of Appendix M10 in~\cite{billingsley1999}. Then $d\big (\tilde {\bf N}_u, \lambda^{\alpha} u^{-\alpha}{\bf p} Y_\alpha(u \, \cdot) \big )$ is a random variable because $d(\cdot, \cdot)$ is continuous on the product space, hence it is measurable.

Finally, we use that $u^{-\alpha} Y_\alpha(u \, \cdot) \stackrel{d}{=} Y_\alpha$ for each fixed $u >0$. Combined with \eqref{eq:Skorokhod in P}, by~Theorem~3.1 in~\cite{billingsley1999} this implies that $\tilde {\bf N}_u \stackrel{d}{\longrightarrow} \lambda^\alpha {\bf p} Y_\alpha$ in $(D([0,T],\R^K), J_1)$, as needed. This theorem applies because the product space is separable, as explained above.
\end{proof}

Finally, we prove Theorem \ref{thm:fclt}. The proof again uses the fact that the inverse $\alpha$-stable subordinator is a common random clock for all independent coordinates. 

\begin{proof}[Proof of Theorem \ref{thm:fclt}] 
By \eqref{eq:Yselfsimilarity}, for each coordinate $M_u^{(i)}$ we have a distributional equality of processes:
\[ 
M_u^{(i)}(t)
   \;=\;
   \frac{N_i(p_i Y_{\alpha}(u t))-p_i \lambda^{\alpha} Y_\alpha(u t)}{u^{\alpha/2}} \stackrel{d}{=} \frac{N_i(p_i  u^\alpha Y_{\alpha}(t))-p_i u^\alpha \lambda^{\alpha} Y_\alpha(t)}{u^{\alpha/2}}
\]
for every $u>0$. Even more, define the processes 
\[
L_u^{(i)}(s) = \frac{N_i(p_i u^{\alpha}s)- p_i \lambda^\alpha u^{\alpha}s}{u^{\alpha/2}}, \qquad s \ge 0,
\]
then, since the processes $N_1, \ldots, N_K$, $Y_\alpha$ are independent, it again follows from  \eqref{eq:Yselfsimilarity} that
\[
(M_u^{(1)}, \ldots, M_u^{(K)}) \stackrel{d}{=} (L_u^{(1)}, \ldots, L_u^{(K)}) \circ Y_{\alpha},
\]
for every $u>0$. 

We will need the following fact about continuity of the composition mapping in the Skorokhod topology $J_1$, see Theorem~13.2.2 in~\cite{Whitt2002}. Let $\{x_k\}_{k \ge 1}$ and $x$ be functions in $D([0, \infty), \R^K)$, and let $\{y_k\}_{k \ge 1}$ and $y$ be in $D_\uparrow$, the set of non-negative non-decreasing functions in $D([0, \infty), \R)$.
If $(x_k, y_k) \to (x, y)$ in $J_1 \times J_1$ as $k \to \infty$ and  $x$ is a continuous function, then $x_k \circ y_k \to x \circ y$.

We now use that 
\[
\big( (L_u^{(1)}, \ldots, L_u^{(K)}) , Y_{\alpha} \big)  \stackrel{d}{\longrightarrow} \big(  (\sqrt{p_1\lambda^\alpha}B_1, \ldots, \sqrt{p_K\lambda^\alpha}B_K), Y_{\alpha} \big)
\]
in $D([0, \infty), \R^K) \times D_\uparrow$ equipped with $J_1 \times J_1$ as $u \to \infty$. Here the weak convergence of the first $K$ coordinates to the $K$-dimensional Brownian motion follows from the classical functional central limit theorem (FCLT), and the last coordinate is added using independence of the $K+1$ coordinate processes on the l.h.s. The limit process a.s.\ belongs to the set $C([0,\infty), \R^K) \times D_\uparrow$, which is contained in the set of continuity points of the composition mapping. Therefore, the continuous mapping theorem implies that
\[
(M_u^{(1)}, \ldots, M_u^{(K)}) = ((L_u^{(1)}, \ldots, L_u^{(K)}) \circ Y_{\alpha}) \stackrel{d}{\longrightarrow}
(\sqrt{p_1\lambda^\alpha}B_1, \ldots, \sqrt{p_K\lambda^\alpha}B_K)\circ Y_\alpha.
\]
\end{proof}

\begin{remark} Theorem \ref{thm:fclt} can be proved by first conditioning on the $\sigma$-algebra of the subordinator, use the conditional FCLT for Poisson processes and then remove the conditioning, just as the proof of the last part of Corollary \ref{cor:babyarrival}. For this method one would need to utilise the conditional $\to$ unconditional route in \cite{EK86}. However, the proof above is more elegant and has already been used in similar settings (see for example \cite{LeonenkoScalasTrinh2019}).  
\end{remark}

\section{Proofs for the queueing system}
\label{sec:qproofs}

In this section we prove the results for the restless multiclass queue. In Subsection \ref{subsec:scal} we deal with the first-order scaling limits in the spirit of Theorem 3.3 in \cite{ButtGeorgiouScalas2022}  (see for convenience Appendix \ref{app:A}, theorem \ref{thm:33}) and an FCLT as corollary of a compensated queue length. 
In Subsection \ref{subsec:rec}, we prove that the restless (multiclass) queue empties infinitely often when the indices of the arrival and service FPPs are equal. 

\subsection{Scaling limits for queue lengths}
\label{subsec:scal}
\begin{proof}[Proof of Theorem \ref{thm:QConvergence}] 
The queue length of the first $i$ classes is given by 
\be\label{eq:pbyp}
Q_{\le i}(t)= N^{\alpha,\lambda}_{\le i} (t) - D^{\beta, \mu}(t) - \inf_{s \le t}\bigg \{N^{\alpha,\lambda}_{\le i}(s) - D^{\beta, \mu}(s) \bigg\}.
\ee

By \eqref{eq:ThinSum0}, the aggregated arrival process $N^{\alpha,\lambda}_{\le i} $ above is an FPP$ (\alpha, \lambda P_i^{1/\alpha} )$. Therefore, equation~\eqref{eq:pbyp} is the same as equation (3.19) in \cite{ButtGeorgiouScalas2022}, with FPP counting processes. Therefore, Theorem~3.3 in~\cite{ButtGeorgiouScalas2022}  (see for convenience Appendix \ref{app:A}, theorem \ref{thm:33}) applies and we obtain~\eqref{eq:Qilimit} as its direct corollary. 

Consider now the individual queue length $Q_i$ for Class $i$ when $\alpha=\beta$. It is given by $Q_{\le i} - Q_{\le i-1}$ if we set $Q_{\le 0}(t)\equiv 0 $. For $u > 0$ and $t \ge 0$, denote 
\[
N^{\alpha, \lambda}_{\le i, u}(t)= \frac{1}{u^\alpha} \sum_{j=1}^i N^{(i)}(ut), \qquad D_u(t)= u^{-\alpha} D^{\alpha, \mu}(ut).
\]
It then follows from~\eqref{eq:pbyp} that in terms of the Skorokhod reflection map $\Phi$,
\[
\frac{Q_i(ut)}{u^\alpha}=\frac{Q_{\le i}(ut)}{u^\alpha} - \frac{Q_{\le i-1} (ut)}{u^\alpha} = \Phi(N^{\alpha, \lambda}_{\le i,u} -  D_u)(t)- \Phi(N^{\alpha, \lambda}_{\le i-1,u} - D_u)(t).
\]
In terms of the normalized time-changed multivariate class count process 
\[
{\bf N}_u(t)=\frac{1}{u^{\alpha}}(N^{(1)}(ut),  \ldots , N^{(K)}(ut)),
\]
the above becomes
\begin{equation} \label{eq: Q_i = F}
\frac{Q_i(u t)}{u^\alpha} = F ({\bf N}_u, D_u)(t),
\end{equation}
where $F$ is a mapping from $D([0, \infty), \R^K) \times D([0, \infty), \R)$ to $ D([0, \infty), \R)$ defined by
\[
F((f_1, \ldots, f_K),g)= \Phi(f_1 + \ldots + f_i - g) - \Phi(f_1 + \ldots + f_{i-1} - g)
\]
for $f_1, \ldots, f_K, g$ in  $D([0, \infty), \R)$.

In general, addition and subtraction are not continuous on  $D([0, \infty), \R^K)$ equipped with $J_1$. However, they are so at every point $(f, g)$ such that $f$ and $g$ do not have common jump times. The coordinate projections are continuous, and the Skorokhod reflection map is continuous when $K=1$, see Theorem~13.5.1 in~\cite{Whitt2002}. Also, $\Phi$ maps continuous functions into continuous ones. All together, this gives that the map $F$ is continuous with respect to $J_1 \times J_1$ at every point of $C([0, \infty), \R^K) \times C([0, \infty), \R)$. This set supports the trajectories of the process $(\lambda^\alpha {\bf p} Y_\alpha, \mu^\alpha \widetilde{Y}_\alpha)$, which by Theorem~\ref{thm:LLN} (applied twice) is the weak limit of $({\bf N}_u, D_u)$ as $u \to \infty$ in $D([0, \infty), \R^K) \times D([0, \infty), \R)$ equipped with $J_1 \times J_1$. Therefore, equation~\eqref{eq: Q_i = F} and the continuous mapping theorem imply that the processes $u^{-\alpha} Q_i(u \, \cdot)$ converge weakly to $F(\lambda^\alpha {\bf p} Y_\alpha, \mu^\alpha \widetilde{Y}_\alpha)$. This proves \eqref{eq:individualq}.
\end{proof}

\begin{proof}[Proof of Corollary \ref{cor:qclt}] For any $t, u \ge 0$, denote
\be \label{eq:ABC}
A_u(t) = \sum_{j=1}^i N_j( p_j Y_\alpha(ut))  - \lambda^\alpha P_i Y_\alpha(ut),\quad  C_u(t) = \widetilde N(\widetilde Y_{\beta}(ut)) - \mu^\beta \widetilde Y_\beta(ut).
\ee

From the definition in \eqref{Not1}, we have by rearranging terms that 
\[
Q^{\rm cent}_{\le i}(ut)/u^{\gamma/2}= \Phi \big( N_{\alpha, \beta}^{\rm diff} - Y_{\alpha, \beta}^{\rm diff} \big)(ut)/u^{\gamma/2} = \Phi \big(A_u/u^{\gamma/2}- C_u/u^{\gamma/2} \big)(t).
\]
Then by Theorem \ref{thm:fclt} and independence of $A_u$ and $C_u$, we have the joint weak convergence 
\be \label{eq:fclt2}
\Big ( \frac{A_u}{u^{\alpha/2}} , \frac{C_u}{u^{\beta/2}} \Big) \stackrel{d}{\longrightarrow} \big(  { B}(\lambda^\alpha P_i X_\alpha(\cdot)), \widetilde {B}(\mu^{\beta} \widetilde X_\beta( \cdot)) \big)
\ee
in $J_1 \times J_1$  as $u \to \infty$. Recalling that $\gamma = \max(\alpha, \beta)$, in the case where $\alpha > \beta$ we have 
\[
\frac{C_u}{u^{\gamma/2}} = \frac{C_u}{u^{\beta/2}} \frac{1}{u^{(\alpha - \beta)/2}} \stackrel{d}{\longrightarrow} 0,
\] 
and therefore in this case,
\be \label{eq:wrongscale}
\Big ( \frac{A_u}{u^{\gamma/2}} , \frac{C_u}{u^{\gamma/2}} \Big) \stackrel{d}{\longrightarrow} \big( B(\lambda^\alpha P_i X_\alpha(\cdot)), 0 \big).
\ee

Since the Skorokhod map $\Phi$ is continuous in $J_1$ and the limit processes in \eqref{eq:fclt2} and \eqref{eq:wrongscale} have continuous trajectories a.s., for $\alpha=\beta$ and $\alpha>\beta$ the corollary follows by the continuous mapping theorem applied as in the proof of Theorem \ref{thm:QConvergence}. For $\alpha < \beta$, by the same argument and the symmetry of Brownian motion, 
\[
\frac{Q^{\rm{cent}}_{\le i}(u \, \cdot)}{u^{\beta/2}}\stackrel{d}{\longrightarrow} 
\Phi(-\widetilde {B}(\mu^{\beta} \widetilde X_\beta(\cdot))) \stackrel{d}{=} \Phi(\sqrt{\mu^\beta} \widetilde B \circ \widetilde X_\beta).
\]
\end{proof}

\subsection{Recurrence and transience.}
\label{subsec:rec}
\begin{proof}[Proof of Theorem \ref{thm:Rec+Trans}]
We need to show that 
\be \label{eq:Inf1}
\varliminf_{t \to \infty} [N^{\alpha, \lambda}(t) - D^{\alpha, \mu}(t)]= -\infty  \quad{\rm a.s.}.
\ee
After that the first claim in the theorem follows, since the infimum process only decreases by 1 at the times when a service was completed but the queue was already empty, by~\eqref{eq:fullqueuelength} and Lemma 6.2. in \cite{ButtGeorgiouScalas2022} (see Theorem \ref{thm:62} in Appendix \ref{app:A}).

From \eqref{eq:EquivReps}, we see that \eqref{eq:Inf1} is equivalent to 
\[
\varliminf_{t \to \infty} [N(\lambda^\alpha Y_\alpha(t)) - \tilde N(\mu^\alpha \tilde Y_\alpha(t))]= -\infty \quad { \rm a.s.},
\]
where $N$ and $\tilde N$ are standard rate 1 Poisson processes, $Y_\alpha$ and $\tilde Y_\alpha$ are the inverses of the respective standard $\alpha$-stable subordinators $L_\alpha$ and $\tilde L_\alpha$, and the four processes $N$, $\tilde N$, $L_\alpha$, $\tilde L_\alpha$ are jointly independent. 

For any standard Poisson process, say $N$, the law of large numbers gives that a.s.
\be\label{eq:es}
\frac{1}{2} < \lim_{t\to \infty} \frac{N(t)}{t} = 1 < 2.
\ee
Since  $Y_\alpha(t) \to \infty$ and $\tilde Y_\alpha(t) \to \infty$ as $t \to \infty$, from the law of large numbers applied for $N$ and $\tilde N$, we see from the inequalities in \eqref{eq:es} that
\be \label{eq:office}
\varliminf_{t \to \infty} [N(\lambda^\alpha Y_\alpha(t)) - \tilde N(\mu^\alpha \tilde Y_\alpha(t))] \le \varliminf_{t \to \infty} [2\lambda^\alpha Y_\alpha(t)  - \mu^\alpha \tilde Y_\alpha(t)/2] \text{ a.s.}
\ee

Therefore, it suffices to show that for every $c>0$,
\begin{equation} \label{eq: liminf}
\varliminf_{t \to \infty} [Y_\alpha(t)  - c \tilde Y_\alpha(t)] = -\infty \text{ a.s.}
\end{equation}

Consider the sequence 
\[ 
S_n = L_\alpha(n)-\tilde L_\alpha(n/c) \quad {\rm for \,\,} n\in \N_0.
\] 
Its differences $J_{n+1} = S_{n+1} - S_n$ satisfy $J_{n+1}=W_{n+1} + \widetilde W_{n+1}$, where
\[  W_{n+1} = L_\alpha(n+1) -L_\alpha(n), \qquad \widetilde W_{n+1} = \tilde L_\alpha(n/c) -\tilde L_\alpha((n+1)/c).
\]
The sequences $\{W_n\}_n$ and $\{\widetilde W_{n}\}_n$ are i.i.d.\ respectively as increments of L\'evy processes over fixed size intervals, and jointly independent. 
These random variables are strictly $\alpha$-stable, and therefore for any $a,b>0$,
\[ 
a J_1 + b J_2 =(a W_1 + b  W_2 )  + (a \widetilde W_1 + b \widetilde W_2 ) \stackrel{d}{=} (a^\alpha + b^\alpha)^{1/\alpha} W_1 + (a^\alpha + b^\alpha)^{1/\alpha} \widetilde W_1 =  (a^\alpha + b^\alpha)^{1/\alpha} J_1. 
\]
This makes $S_n$ a random walk whose i.i.d.\ increments $J_n$ are strictly stable and their common distribution puts mass on both half-axes $(-\infty, 0)$ and $(0, \infty)$, that is $\P(J_n>0)>0$ and $\P(J_n<0)>0$. 

Such random walks are known to oscillate, that is
\[
\varlimsup_{n \to \infty} [L_\alpha(n)-\tilde L_\alpha(n/c)] = +\infty \text{ a.s.}, \qquad \varliminf_{n \to \infty} [L_\alpha(n)-\tilde L_\alpha(n/c)] = -\infty \text{ a.s.},
\]
by Example (a) after Theorem XII.7.2 in Feller~\cite{Feller}. Hence the integer random variables, defined by $n_0=0$ and 
\[
n_k =\inf \big \{n: n> n_{k-1}, L_\alpha(n) \ge \tilde L_\alpha(n/c), n \in \N \big\}= \inf \big \{n: n> n_{k-1}, S_n \ge 0, n \in \N \big\}
\]
for $k \in \N$, are all finite a.s. They are stopping times with respect to the sequence $S_n$.

Fix an $M\in\N$. Then for every $k\in \N$, we have
\begin{align*}
\P \big(L_\alpha(n_k+1)-L_\alpha(n_k) \ge \tilde L_\alpha ((n_k+M)/c)-\tilde L_\alpha(n_k/c)\big) &=\P \big(L_\alpha(1) \ge \tilde L_\alpha (M/c)\big) =:a >0
\end{align*}
by the strong Markov property of the processes $L_\alpha(t)$ and $\tilde L_\alpha(t/c)$ and the facts that they are independent and have independent stationary increments. 

Define the integer random variables $K_0=0$ and for $i \in \N$,
\[
K_i=\inf \big \{k: k > K_{i-1}, L_\alpha(n_k+1)-L_\alpha(n_k) \ge \tilde L_\alpha ((n_k+M)/c)-\tilde L_\alpha(n_k/c), k \in \N \big \}.
\]

Let us argue that these variables are finite a.s. 
The first one, $K_1$, has geometric tails: For every integer $k \ge M$, we have
\begin{align*}
\P\{K_1 >k\} &\le \P \Big\{L_\alpha(n_k+1)-L_\alpha(n_k) < \tilde L_\alpha ((n_k+M)/c)-\tilde L_\alpha(n_k/c), \, K_1 >k-M  \Big\}  \\
&\le (1-a) \P\{K_1 >k-M\} \le \ldots \le (1-a)^{[k/M]},
\end{align*}
where we used that the two events under the probability on the r.h.s.\ of the first inequality are independent.
To see this, note that $\{K_1 >k-M\}$ is defined in terms of $\{L_\alpha(t)\}_{t \le n_{k-M}+1}$ and $\{\tilde L_\alpha(t/c)\}_{t \le n_{k-M}+M}$, where $n_{k-M}+1 \le n_{k-M}+M \le n_k$. The other event is defined in terms of the increments of $L_\alpha$ and $\tilde L_\alpha$ after $n_k$, and  therefore it is independent of $\{K_1 >k-M\}$. This shows that $K_1$ is finite a.s.\ since $a \in (0,1)$. Similarly, we obtain that each $K_i$ is finite a.s.

Since $L_\alpha(n_k)\ge \tilde L_\alpha(n_k/c)$ a.s.\ for every $k \in \N$ by the definition of $n_k$, we have
\[
L_\alpha(n_{K_i}+1) \ge \tilde L_\alpha((n_{K_i}+M)/c) \text{ a.s.}
\]
for every $i \in \N$ by the definition of $K_i$. 

Furthermore, both inverse subordinators $Y_\alpha$ and $\tilde Y_\alpha$ are non-decreasing and satisfy $Y_\alpha(L_\alpha(n))=n$ and $\tilde Y_\alpha(\tilde L_\alpha(n/c))=n/c$ for all $n \in \N$ with probability $1$, since the trajectories of any L\'evy processes are continuous at a fixed deterministic point  with probability~$1$. 
In particular, we have a full probability event $\Omega_0$ (partitioned across $n \in \N$) on which  
\[ Y_\alpha(L_\alpha(n))=n \text{  and  } \tilde Y_\alpha(\tilde L_\alpha(n/c))=n/c, \text{ for all } n \in \N.
\]
Then on $\Omega_0$, for every $i \in \N$, 
\[
c\tilde Y_\alpha(L_\alpha(n_{K_i}+1)) \ge c\tilde Y_\alpha(\tilde L_\alpha((n_{K_i}+M)/c))=n_{K_i} + M = Y_\alpha(L_\alpha(n_{K_i}+1))+M-1 \text{ a.s.}
\]
Since $L_\alpha(n_{K_i}+1) \to \infty$ a.s.\ as $i \to \infty$, this implies that
\[
\varliminf_{t \to \infty} [Y_\alpha(t)  - c \tilde Y_\alpha(t)] \le 1-M \text{ a.s.,}
\]
and \eqref{eq: liminf} follows since $M \in \N$ was chosen arbitrarily. This concludes the proof of the first statement in the theorem.  

For the second statement, we have $Q_{\le 1}(t)\le Q_{\le i}(t)$ for all $t \ge 0$, therefore it suffices to prove the claim for $i=1$. 

Then
\begin{align*}
\varlimsup_{t \to \infty} Q_{\le 1}(t) &= \varlimsup_{t \to \infty} \Phi(N^{(1)} - D^{\alpha, \mu})(t) \ge \varlimsup_{t \to \infty} [N^{(1)}(t) - D^{\alpha, \mu}(t)] \\
&\stackrel{d}{=} \varlimsup_{t \to \infty} [N(\lambda^\alpha p_1 Y_\alpha(t)) - \tilde N(\mu^\alpha \tilde Y_\alpha(t))], 
\end{align*}
and by the same reasoning as in \eqref{eq:office}, we also have
\[
\varlimsup_{t \to \infty} [N(\lambda^\alpha p_1Y_\alpha(t)) - \tilde N(\mu^\alpha \tilde Y_\alpha(t))] \ge \varlimsup_{t \to \infty} [\lambda^\alpha p_1 Y_\alpha(t)/2  - 2\mu^\alpha \tilde Y_\alpha(t)] \text{ a.s.}
\]
Finally, using \eqref{eq: liminf} applied with $c=\lambda^\alpha \mu^{-\alpha}p_1 /4 $ together with the distributional equality $Y_\alpha - c\tilde Y_\alpha \stackrel{d}{=} \tilde Y_\alpha - c Y_\alpha$, we get
\[
\varlimsup_{t \to \infty} [\lambda^\alpha p_1 Y_\alpha(t)/2  - 2\mu^\alpha \tilde Y_\alpha(t)] = -2 \mu^\alpha \varliminf_{t \to \infty} [\tilde Y_\alpha(t) - c Y_\alpha(t)] =  + \infty \text{ a.s.}
\]
The second statement in the theorem then follows. 
\end{proof}

A byproduct of this proof is the following corollary. 
\begin{corollary}
\label{cor:Ydiff}
Let $c>0$ and $Y_\alpha, \widetilde Y_{\alpha}$ be independent inverse $\alpha$-stable subordinators. Then with probability~$1$, we have
\[
\varliminf_{t \to \infty}(Y_\alpha(t) - c \widetilde Y_{\alpha}(t)) = -\infty, \quad \text{ and } \quad   \varlimsup_{t \to \infty}(Y_\alpha(t) - c \widetilde Y_{\alpha}(t)) = +\infty.
\]
Moreover, if $N, \widetilde N$ are independent Poisson processes of rate 1 that are independent of $Y_\alpha, \tilde Y_\alpha$ and $\lambda, \mu >0$, then with probability $1$, 
\[
\varliminf_{t \to \infty}(N(\lambda Y_\alpha(t)) -  \widetilde N(\mu\widetilde Y_{\alpha}(t)) = -\infty, \quad \text{ and } \quad   \varlimsup_{t \to \infty}(N(\lambda Y_\alpha(t)) -  \widetilde N(\mu\widetilde Y_{\alpha}(t)) = +\infty.
\]
\end{corollary}

\section{Example: multiclass queueing model with a continuum of classes} 
\label{sec:doubleauction}

Let $\{G_n\}_{n \ge 1}$ be a sequence of i.i.d. random variables on 
$\R_+$ having a common distribution function 
\[ F_G(x) = \P\{ G \le x\}. \]
{\em A priori} no assumptions are made on $G$ except that there is no mass at 0; it does not need to be atomless anywhere else, or have moments, or be unbounded. 

Then we define the following multiclass fractional model. Let 
$N^{\alpha, \lambda}(t)$ denote the FPP that signifies an arrival. When the $k$-th event of $N^{\alpha, \lambda}(t)$ occurs, a customer appears at location $G_k$.
We say that every customer at location $x \in \R_+$ has class $C_x$ and classes are ordered from smallest to largest at any given time. 
If $G$ is not atomless then there is probability that multiple customers accumulate at the atoms of the distribution - in that case they all belong to the same class. 
When a removal event from the departure process $D^{\beta, \mu}(t)$ occurs, the earliest customer from the smallest class (i.e.\ the one closest to zero) leaves the system. 

\begin{remark} This model is a prototypical, simplified one-sided, continuous double-auction model. The locations of customers in the queue correspond to the locations of `asks' for the auction -- those are prices for which an asset can be bought if someone is willing to pay the price. However, as asks come in, the cheapest assets are sold first. In this spatial multiclass queue, a successful ask corresponds to a removal of the leading customer. As such, the location of the front customer corresponds to the ``best ask" value, and it is a quantity of interest. 
\end{remark}

Let
\[Q_t(x) = \text{number of customers in $(0,x]$ at time $t$}.\]
Then $Q_t(\infty)$ is the total length of the queue, which is still given by the r.h.s.\ of \eqref{eq:fullqueuelength}. Therefore, the scaling theorems about the total queue length or questions about whether the whole queue empties infinitely often, follow from \cite{ButtGeorgiouScalas2022}.

We denote the queue length in any interval $(a, b]$ or $(a, \infty)$ by 
\[
Q_t(a,b) = \text{number of customers in $(a,b]$ at time $t$}.
\]
For any finite partition $ V = \{ c_1, c_2, \ldots, c_{K-1} \}$ of $\R_+$, define probabilities 
\[ 
p^V_1 = \P\{ G \in (0, c_1]\}, \quad  p^V_2 = \P\{ G \in (c_1, c_2]\}, \quad  \ldots,\quad p^V_K = \P\{ G \in (c_{K-1}, \infty)\}. 
\]

Then the arrival processes in each of these sets are also thinned FPPs since we are looking at a discrete set of classes now. In particular, the arrival processes to these (arbitrarily chosen) classes 
\[
(N^{(1)}(t), N^{(2)}(t),  \ldots, N^{(K)}(t)), \quad N^{(i)}(t) \sim { \rm FPP}(\alpha, \lambda (p^V_i)^{1/\alpha}),
\]
are distributed as in Theorem \ref{thm:babyarrival} and satisfy also Theorems \ref{thm:LLN} and \ref{thm:fclt}.
In this setting, a question of interest is the ``best ask" value, i.e. the location of the particle nearest zero at time $t$. That is the smallest class in the system. Define $\mathcal C_{\min}(t)$ by
\[
 C_{\min}(t) = \sup \{ x \in \R_+: Q_t(x) = 0\}.
\]
In the cases where $\alpha \le \beta$ we know that the whole queue empties infinitely often by Theorem \ref{thm:Rec+Trans} (when $\alpha = \beta$) and Theorem 3.4 in \cite{ButtGeorgiouScalas2022} (when $\alpha < \beta$). This gives us a sequence of random times $T_n \to \infty$ such that $C_{\min}(T_n) = \infty$ for all $n \in \N$. 

For this example, we will only show that 
\be
C_{\min}(t) \stackrel{\bP}{\longrightarrow} \inf \{x  \ge 0: F_G(x)>0\}
\ee
as $t \to \infty$ in the case of $\alpha > \beta$.

To see this, denote $a:=\inf \{x: F_G(x)>0\}$. Then $F_G(a+\epsilon)= :p_\epsilon> 0$ for all $\epsilon>0$. Consider the arrival process $N^{\alpha, \lambda}_{\epsilon}(t)$ in the interval $[0, a+\epsilon]$. By Theorem \ref{thm:babyarrival} we have 
\[
N^{\alpha, \lambda}_{\epsilon}(t) \stackrel{d}{=} N_1(p_\epsilon \lambda^\alpha Y_{\alpha}(t)) \sim {\rm FPP}(\alpha, \lambda p_{\epsilon}^{1/\alpha}). 
\]
Then we can estimate 
\begin{align*} 
\varlimsup_{t\to \infty}\P\{ C_{\min}(t) > \epsilon \} &= \varlimsup_{t\to \infty}\P\{ Q_t(a+\epsilon) =0 \} \le \varlimsup_{t\to \infty}\P\{ Q_t(a+\epsilon) \le t^{\beta/2} \}\\
&\le \varlimsup_{t\to \infty}\P\{ N^{\alpha, \lambda}_{\epsilon}(t) - D^{\beta, \mu}(t)  \le t^{\beta/2} \} \\
&= \varlimsup_{t\to \infty}\P\{ e^{sN^{\alpha, \lambda}_{\epsilon}(t) - sD^{\beta, \mu}(t)}  \ge e^{s t^{\beta/2}} \}
\end{align*}
for any $s<0$. By \eqref{eq:mgf} and the exponential Markov inequality, 
\begin{align*}
\varlimsup_{t\to \infty}\P\{ C_{\min}(t) > \epsilon \} & \le \varlimsup_{t\to \infty}  e^{-st^{\beta/2}}E_{\alpha}((e^{s}-1)\lambda^\alpha p_{\epsilon}t^{\alpha}) E_{\beta}((e^{-s}-1)\mu^{\beta}t^\beta).  
\end{align*}
We can put $s = - t^{-\beta}$ and get the bound 
\begin{align*}
\varlimsup_{t\to \infty}\P\{ C_{\min}(t) > \epsilon \} & \le \varlimsup_{t\to \infty}  e^{t^{-\beta/2}}E_{\alpha}((e^{-t^{-\beta}}-1)\lambda^\alpha p_{\epsilon}t^{\alpha}) E_{\beta,1}((e^{t^{-\beta}}-1)\mu^{\beta}t^\beta)\\
& \le \varlimsup_{t\to \infty}  e^{t^{-\beta/2}}E_{\alpha}(-\lambda^\alpha p_{\epsilon}t^{\alpha -\beta} +o(t^{\alpha-\beta})) E_{\beta}(\mu^{\beta} + o(1)) = 0.
\end{align*}
This suffices for the convergence in probability. 

Moreover, the estimates above imply that 
$ \lim_{t\to \infty}\P\{ Q_t(a+ \epsilon) > t^{\beta/2} \} =1$. This not only suggests that the customers are dense around 0 with high probability, but also implies that 
\[ Q_t(a+\epsilon) \stackrel{\P}{\longrightarrow} + \infty \quad \text{as } t \to \infty.\]

\appendix 

\section{Auxiliary results}
\label{app:A}

In this small appendix we recall the three theorems from  \cite{ButtGeorgiouScalas2022} that are used here, using the notation of this article. Also note, at the time we called the queueing model 3 in \cite{ButtGeorgiouScalas2022} the GI/GI/1 queue, but after the classifications of article \cite{AscioneCaputo}, we are calling queues with renewal-type departures and the possibility of completion of empty services, as `restless'. 

In this respect, the results below are for the single-class, restless Mittag-Leffler queue. The single-class queue length is denoted by  $Q^{\alpha, \beta}_{\lambda, \mu}$.

\begin{theorem}[Theorem 3.3. in  \cite{ButtGeorgiouScalas2022}]\label{thm:33}
Let $Y_{\alpha}(t)$ and $\widetilde Y_{\alpha}(t)$ be two independent inverse standard $\alpha$-stable subordinators, and let $\gamma = \max\{\alpha, \beta\}$.
Let  $\Phi$ be given by~\eqref{eq:Phi}. Then we have the weak convergence
\begin{align}
\left\{ \frac{Q^{\alpha, \beta}_{\lambda, \mu}(u t)}{u^\gamma} \right \}_{t \ge 0}&\stackrel{d}{\longrightarrow} \displaystyle 
\begin{cases} 
\displaystyle \lambda^\alpha  Y_{\alpha}, & \quad \gamma = \alpha >\beta,\\
0, &\quad \gamma = \beta > \alpha, \\ 
\displaystyle  \Phi\big(\lambda^\alpha  Y_\alpha - \mu^\alpha \widetilde Y_\alpha \big), &\quad \gamma = \alpha = \beta,
\end{cases}
\label{eq:Qilimitapp}
\end{align}
in the space $D[0, \infty)$ equipped with the Skorokhod topology $J_1$, as $u \to \infty$ (above, $0$ is as a function). 
\end{theorem}

\begin{theorem}[Theorems 3.4. and 3.5. in  \cite{ButtGeorgiouScalas2022}] \label{thm:3435}
Denote the single-class queue length by  $Q^{\alpha, \beta}_{\lambda, \mu}(t)$. Then for any $\lambda, \mu >0$ we have that 
\begin{enumerate}
	\item If $\alpha < \beta$ then $ \P\{ Q^{\alpha, \beta}_{\lambda, \mu}(t) = 0 \text{ \, i.o.} \} = 1.$
	\item If $\alpha > \beta$ then $ \P\{ \varlimsup_{t \to \infty} Q^{\alpha, \beta}_{\lambda, \mu}(t) = +\infty \} = 1.$
\end{enumerate} 

Moreover, if $\alpha =\beta$, then 
\begin{enumerate}
	\item If $\lambda \le \mu$ then $ \P\{ Q^{\alpha, \beta}_{\lambda, \mu}(t) = 0 \text{ \, i.o.} \} = 1.$
	\item If $\lambda \ge \mu$ then  $ \P\{ \varlimsup_{t \to \infty} Q^{\alpha, \beta}_{\lambda, \mu}(t) = +\infty \} > c_{\lambda, \mu},$ for some positive constant $c_{\lambda, \mu}$.
\end{enumerate} 
\end{theorem}

\begin{remark} Theorem \ref{thm:Rec+Trans} generalises the second part of Theorem \ref{thm:3435} above (alternatively, it generalises Theorem 3.5. of \cite{ButtGeorgiouScalas2022}. When $\alpha=\beta$ then irrespective of the values of $\lambda, \mu$ the queue empties infinitely often with probability 1. Moreover,  $c_{\lambda, \mu} = 1$, again irrespectively of $\lambda, \mu$. 
\end{remark} 

\begin{theorem}
[Lemma 6.2. in \cite{ButtGeorgiouScalas2022}]\label{thm:62}
The following are equivalent 
\begin{enumerate}
	\item Time $T$ is a discontinuity of the infimum process, i.e. 
	\[ \inf_{0 \le s \le T} \{ N^{\alpha, \lambda}(s) - D^{\beta, \mu}(s) \} = \inf_{0 \le s < T} \{ N^{\alpha, \lambda}(s) - D^{\beta, \mu}(s) \}-1 \]
	
	\item The departure process $D^{\beta, \mu}(s)$ has a renewal point at time $T$ and the queue length $Q^{\alpha, \beta}_{\lambda, \mu}(t)$ satisfies 
	\[
	Q^{\alpha, \beta}_{\lambda, \mu}(T-) = Q^{\alpha, \beta}_{\lambda, \mu}(T) = 0.
	\]
\end{enumerate}
In other words, at time $T$ there was an unused service time if and only if the infimum process jumped down.
\end{theorem}

%

\begin{thebibliography}{99}




\bibitem{AscioneCaputo}
G.~Ascione and L.~Caputo.
\newblock Scaling limits for some Mittag-Leffler queues.
\newblock arXiv:2602.16521. Accepted to {\em Theory of Probability and Mathematical Statistics.}

\bibitem{AscioneLeonenkoPirozzi2018}
G.~Ascione, N.~Leonenko, and E.~Pirozzi.
\newblock Fractional queues with catastrophes and their transient behaviour.
\newblock {\em Mathematics} \textbf{6}(9) (2018), 159.

\bibitem{AscioneLeonenkoPirozzi2020SPA}
G.~Ascione, N.~Leonenko, and E.~Pirozzi.
\newblock Fractional Erlang queues.
\newblock {\em Stochastic Processes and their Applications} \textbf{130} (2020), 3249--3276.

\bibitem{AscioneLeonenkoPirozzi2021Chapter}
G.~Ascione, N.~Leonenko, and E.~Pirozzi.
\newblock On the transient behaviour of fractional M/M/$\infty$ queues.
\newblock In {\em Nonlocal and Fractional Operators}, Springer, 2021, pp.~1--22.

\bibitem{Asmussen2003}
S.~Asmussen.
\newblock {\em Applied Probability and Queues}, 2nd ed.
\newblock Springer, New York, 2003.

\bibitem{BeghinMacci2016}
L.~Beghin and C.~Macci.
\newblock Multivariate fractional Poisson processes and compound sums.
\newblock {\em Advances in Applied Probability} \textbf{48}(3) (2016), 691--711.
\newblock doi:10.1017/apr.2016.23.

\bibitem{BeghinMacci2017}
L.~Beghin and C.~Macci.
\newblock Asymptotic results for a multivariate version of the alternative fractional Poisson process.
\newblock {\em Statistics \& Probability Letters} \textbf{129} (2017), 260--268.
\newblock doi:10.1016/j.spl.2017.06.009.

\bibitem{billingsley1999} 
P. Billingsley, Convergence of probability measures, Second edition, John Wiley \& Sons, Inc., New York,
1999.

\bibitem{biolek2025}
D.~Biolek, R.~Garrappa, F.~Mainardi, and M.~Popolizio.
\newblock Derivatives of Mittag-Leffler functions: theory, computation and applications.
\newblock{\em Nonlinear Dynamics} (2025).

\bibitem{ButtGeorgiouScalas2022}
J.~Butt, N.~Georgiou, and E.~Scalas.
\newblock Queuing models with Mittag-Leffler inter-event times.
\newblock {\em Fractional Calculus and Applied Analysis} \textbf{26}(4) (2023), 1465--1503.

\bibitem{CahoyPolitoPhoha2015}
D.~O. Cahoy, F.~Polito, and V.~V. Phoha.
\newblock Transient behavior of fractional queues and related processes.
\newblock {\em Methodology and Computing in Applied Probability} \textbf{17} (2015), 739--759.

\bibitem{ChenYao2001}
H.~Chen and D.~D. Yao.
\newblock {\em Fundamentals of Queueing Networks: Performance, Asymptotics, and Optimization}.
\newblock Springer, 2001.


\bibitem{erde}
A. Erdélyi, W. Magnus, F. Oberhettinger, and F. G. Tricomi, \newblock{Higher Transcendental Functions, vol. 3,
McGraw-Hill, New York, NY, USA, 1955.}

\bibitem{EK86} 
S.\,N.~Ethier and T.\,G.~Kurtz.
  \emph{Markov Processes: Characterization and Convergence}. Wiley, 1986.

\bibitem{Feller} 
W. Feller, An introduction to probability theory and its applications. Vol. II, Second edition. John Wiley \& Sons, Inc., New York-London-Sydney, 1971.

\bibitem{FossKorshunovZachary2011}
S.~Foss, D.~Korshunov, and S.~Zachary.
\newblock {\em An Introduction to Heavy-Tailed and Subexponential Distributions}.
\newblock Springer, 2011.

\bibitem{gnedenko1968} 
B.V. Gnedenko and I.N. Kovalenko, Introduction to Queueing Theory, Israel Program for Scientific Translations, Jerusalem, 1968.

\bibitem{GorenfloKilbasMainardiRogosin2014}
R.~Gorenflo, A.~A. Kilbas, F.~Mainardi, and S.~Rogosin.
\newblock {\em Mittag-Leffler Functions, Related Topics and Applications}, 2nd ed.
\newblock Springer, 2014.


\bibitem{Laskin2003}
N.~Laskin.
\newblock Fractional Poisson process.
\newblock {\em Communications in Nonlinear Science and Numerical Simulation} \textbf{8} (2003), 201--213.


\bibitem{LelandTaqquWillingerWilson1994}
W.~E. Leland, M.~S. Taqqu, W.~Willinger, and D.~V. Wilson.
\newblock On the self-similar nature of Ethernet traffic (extended version).
\newblock {\em IEEE/ACM Transactions on Networking} \textbf{2}(1) (1994), 1--15.

\bibitem{LeonenkoScalasTrinh2019}
N.~Leonenko, E.~Scalas, and M.~Trinh.
\newblock Limit theorems for the fractional nonhomogeneous Poisson process.
\newblock {\em Journal of Applied Probability} \textbf{56}(1) (2019), 246--264.
\newblock doi:10.1017/jpr.2019.16.

\bibitem{mainardi2004} 
F. Mainardi, R. Gorenflo, E. Scalas.
\newblock{\em Fractional generalization of the Poisson process.}
\newblock{Vietnam Journal of Mathematics} {\bf 32} SI, 53--64 (2004).

\bibitem{MeerschaertNaneVellaisamy2011}
M.~M. Meerschaert, E.~Nane, and P.~Vellaisamy.
\newblock{\em The fractional Poisson process and the inverse stable subordinator.}
\newblock {Stochastic Processes and their Applications} \textbf{121}(5) (2011), 1105--1121.

\bibitem{MeerschaertSikorskii2012}
M.~M. Meerschaert and A.~Sikorskii.
\newblock {\em Stochastic Models for Fractional Calculus}.
\newblock De Gruyter, 2012.

\bibitem{OrsingherPolito2011}
E.~Orsingher and F.~Polito.
\newblock On a fractional linear birth--death process.
\newblock {\em Bernoulli} \textbf{17} (2011), 114--137.

\bibitem{PaxsonFloyd1995}
V.~Paxson and S.~Floyd.
\newblock Wide-area traffic: The failure of Poisson modeling.
\newblock {\em IEEE/ACM Transactions on Networking} \textbf{3}(3) (1995), 226--244.

\bibitem{Pillai1990}
Pillai, R. N.
\newblock On Mittag-Leffler functions and related distributions.
\newblock {\em Ann. Inst. Statist. Math.} \textbf{42}(1) (1990), 157–-161.

\bibitem{RabertoScalasMainardi2002}
M.~Raberto, E.~Scalas, and F.~Mainardi.
\newblock Waiting-times and returns in high-frequency financial data: an empirical study.
\newblock {\em Physica A: Statistical Mechanics and its Applications} \textbf{314}(1--4) (2002), 749--755.


\bibitem{Robert2003}
P.~Robert.
\newblock {\em Stochastic Networks and Queues}.
\newblock Springer, 2003.

\bibitem{SabatelliKeatingDudleyRichmond2002}
L.~Sabatelli, S.~Keating, J.~Dudley, and P.~Richmond.
\newblock Waiting time distributions in financial markets.
\newblock {\em The European Physical Journal B -- Condensed Matter and Complex Systems} \textbf{27} (2002), 273--275.

\bibitem{scalas2004} 
E. Scalas, R. Gorenflo and F. Mainardi, Uncoupled continuous-time random walks: Solution and limiting behavior of the master equation, Phys. Rev. E {\bf 69}, 011107 (2004).

\bibitem{VanMieghem2020ML}
P.~Van Mieghem.
\newblock The Mittag-Leffler function.
\newblock {\em arXiv:2005.13330 [math.FA]}, 2020. 
\newblock doi:10.48550/arXiv.2005.13330.

\bibitem{Whitt2002}
W.~Whitt.
\newblock {\em Stochastic-Process Limits}.
\newblock Springer, New York, 2002.

\bibitem{Wiman1905Fund}
A.~Wiman,
\newblock Über den Fundamentalsatz in der Theorie der Funktionen $E_{\alpha}(x)$,
\newblock \emph{Acta Mathematica} \textbf{29} (1905), 191--201. 

\end{thebibliography}

\end{document}